\documentclass[12pt]{article}

\usepackage[latin1]{inputenc}
\usepackage[french, english]{babel}
\usepackage[T1]{fontenc}
\usepackage{tikz}
\usepackage{amssymb}
\usepackage{amsmath}
\usepackage{amsthm}
\usepackage{array}
\usepackage{mathtools}
\usepackage{mathrsfs}
\usepackage{amsfonts}
\usepackage{graphicx}
\usepackage{fancyhdr}
\usepackage{color}
\usepackage{caption}
\usepackage{enumerate}
\usepackage[hypcap]{caption}
\newtheorem{defi}{Definition}
\newtheorem{thm}{Theorem}
\newtheorem{pro}{Proposition}
\newtheorem{lem}{Lemma}
\newtheorem{rmk}{Remark}

\newtheorem{nota}{Notation}

\newcommand{\R}{\mathbb{R}}
\newcommand{\N}{\mathbb{N}}

\begin{document}

\pagestyle{fancy}
\rhead{}
\lhead{}
\chead{}
\rfoot{}
\cfoot{\thepage}
\lfoot{}

\renewcommand{\headrulewidth}{0pt}
\renewcommand{\footrulewidth}{0pt}

\title{Mass transportation on sub-Riemannian structures of rank two in dimension four}
\author{Z. BADREDDINE\thanks{Université Côte d'Azur, Inria, CNRS, LJAD, France; Université de Bourgogne, Institut de Mathématiques de Bourgogne, France}}
\date{}
\maketitle

\begin{abstract}
This paper is concerned with the study of the Monge optimal transport problem in sub-Riemannian manifolds where the cost is given by the square of the sub-Riemannian distance. Our aim is to extend previous results on existence and uniqueness of optimal transport maps to cases of sub-Riemannian structures which admit many singular minimizing geodesics.  We treat here the case of sub-Riemannian structures of rank two in dimension four.
\end{abstract}

\section*{Introduction}

Let $M$ be a smooth connected manifold without boundary of dimension $n\geq 2$. The problem of optimal transportation, raised by Monge~\cite{Mo81} in $1781$, was concerned with the transport of a pile of soil into an excavation. Given two probability measures $\mu, \nu$ on $M$, we call the transport map from $\mu$ to $\nu$, any measurable application $T:M\rightarrow M$ such that $T_{\sharp} \mu=\nu$ (we say that $T$ is pushing forward $\mu$ to $\nu$, ie. for every measurable set $B$ in $M$, $\mu(T^{-1}(B))= \nu (B)$).Therefore, the Monge problem was modelized as an optimal transport problem consisting in minimizing the transportation cost
$$
\displaystyle{\int_{M} c(x, T(x)) \mathrm{d}\mu(x)},
$$
among all the transport maps $T:M \rightarrow M$. 

\begin{center}
\begin{figure}
\includegraphics[scale=0.5]{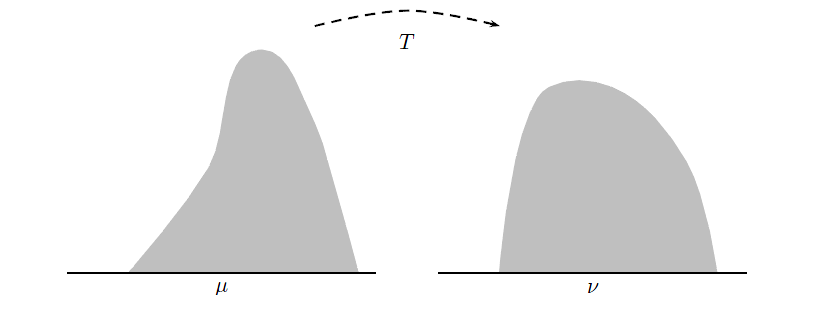}
\end{figure}
\end{center}

Here, $c(x,y)$ represents the cost of transporting a unit of mass from a position $x$ to some position $y$. The fact that the condition $T_{\sharp}\mu =\nu$ is nonlinear with respect to $T$, is the main difficulty in solving the Monge problem. \\

In $1942$, Kantorovitch~\cite{Ka42} proved a duality theorem to study the relaxed form of the problem. He replaced the transport map $T:M\rightarrow M$ by a transport plan $\alpha \in \Pi(\mu, \nu)$ where $\Pi(\mu, \nu)$ is the set of probability measures $\alpha$ in the product $M\times M$ with $P^{1}_{\sharp}(\alpha)= \mu$ and  $P^{2}_{\sharp}(\alpha)= \nu$ ( where $P^{i}: M \times M \rightarrow M$ the projection map into the i-th component). Hence, Kantorovitch problem consists in minimizing 
$$ \displaystyle{\int_{M\times M} c(x,y) \mathrm{d}\alpha (x,y)},\text{among all the transport plans $\alpha \in \Pi(\mu, \nu)$.}$$ 
The Kantorovitch's approach leads to a dual formulation (see Chapter 5 \cite{Vil08}) given by:

\begin{multline}\label{pbdual}
\inf_{\alpha \in \Pi(\mu, \nu)} \left \{ \int_{M\times M} c(x,y) \mathrm{d}\alpha(x,y)\right \} =\\
\sup_{\begin{array}{l}
(\varphi, \psi) \in L^{1}(\mu)\times L^{1}(\nu)\\
 \psi(y)- \varphi(x)\leq c(x,y)
 \end{array}} \left \{ \int_{M} \psi(y) \mathrm{d}\nu(y) - \int_{M} \varphi(x) \mathrm{d}\mu(x)\right \} .
 \end{multline}
This leads to find a pair of integrable functions $(\varphi, \psi)$ optimal on the right-hand side, and a transport plan $\alpha$ optimal on the left-hand side. The pair of functions $(\varphi, \psi)$ should satisfy $\psi(y)- \varphi(x) \leq c(x, y)$. Then, for a given $y$, $\psi(y)$ will be the infinimum of $\varphi(x) + c(x, y)$ among all $x$. For a given $x$, $\varphi(x)$ will be the supremum of $\psi(y)- c(x, y)$ among all $y$. We may indeed assume that $\varphi$ is a c-convex function and $\psi= \varphi^{c}$ satisfying the two equations below:

\begin{subequations}
\begin{gather}
\varphi(x)  = \sup_{y\in M}\Bigl \{ \varphi^{c}(y)- c(x,y)\Bigr \}, \ \forall x \in M \label{kanto1} \\
\varphi^{c}(y)  = \inf_{x\in M}\Bigl \{ \varphi(x) +c(x,y)\Bigr \}, \  \forall y \in M \label{kanto2}
\end{gather}
\end{subequations} 
The pair $(\varphi, \varphi^{c})$ is called the Kantorovitch potentials.\\

We refer the reader to the textbooks ~\cite{Vil03, Vil08} by Villani for more details on the optimal transport theory.
\vspace*{0.5cm}

Several techniques developed by Brenier~\cite{Br91}, McCann~\cite{Mc01}, Cavalletti and Huesmann \cite{CH15} and others allow to show that in certain cases, optimal transport plans yields indeed optimal transport maps, solutions to the Monge problem. \\

This paper will be concerned with the study of the Monge problem for the quadratic geodesic sub-Riemannian cost. Let $(\Delta, g)$ be a complete sub-Riemannian structure on $M$, where $\Delta$ is a totally nonholonomic distribution on $M$ of rank $m$ $(m<n)$ and $g$ a smooth Riemannian metric on $\Delta$, that is for every $x\in M$, $g_{x}$ is a scalar product on $\Delta({x})$. We recall that a distribution $\Delta$ is called totally nonholonomic if, for every $x\in M$, there exist an open neighborhood $\mathcal{V}_{x}$ of $x$ and a local frame $X^{1}_{x}, \dots, X^{m}_{x}$ on $\mathcal{V}_{x}$ such that 
$$
\mbox{Lie}\Bigl \{ X^{1}_{x}, \dots, X^{m}_{x} \Bigr \}(y) = T_{y}M,\ \forall y\in \mathcal{V}_{x}.
$$ 
Let $T>0$. A continuous path $\gamma : [0, T]\rightarrow M$ is said to be horizontal with respect to $\Delta$ if it is absolutely continuous with square integrable derivative and satisfies
$$
\dot{\gamma}(t) \in \Delta (\gamma(t)), \ a.e. \ t\in [0, T].
$$
The length of an horizontal path $\gamma$ is given by 
$$l(\gamma):=\displaystyle{ \int_{0}^{T}\sqrt{g_{\gamma(t)}(\dot{\gamma}(t), \dot{\gamma}(t))} \mathrm{d}t}.$$
We define the sub-Riemannian distance $d_{SR}(x, y)$ between two points $x$ and $y$ of $M$ as the infinimum of lengths of horizontal paths joining $x$ to $y$, that is,
$$d_{SR}(x, y) := \inf \left \{l(\gamma)| \ \gamma:[0,T]\rightarrow M\ \text{horizontal path s.t.}\ \gamma(0)=x, \gamma(T)=y\right \}.$$

A minimizing geodesic is an horizontal path with constant speed minimizing for the sub-Riemannian distance between its end-points. We shall say that the sub-Riemannian structure $(\Delta, g)$ on $M$ is complete if the metric space $(M, d_{SR})$ is complete. Thanks to the Hopf-Rinow theorem (see~\cite{Rif14}), if $(\Delta,g)$ is a complete sub-Riemannian structure on $M$, then minimizing geodesics exist between any pair of points in $M$. Let $\{ X^{1}, \dots, X^{k} \}$ be $k\leq m(n+1)$ smooth vector fields generating $\Delta$ (see proposition $1.1.8$~\cite{Rif14}), that is for every $y\in M$, 
$$
\Delta(y)= \mbox{Span}\Bigl \{ X^{1}(y), \dots, X^{k}(y)\Bigr \}.
$$ 
Given $x\in M$ and $T>0$, the End-point mapping from $x$ is defined by
$$\begin{array}{lccl}
E^{x}:& L^{2}([0,T], \R^{k}) & \rightarrow & M\\
&u& \mapsto &E^{x}(u) = \gamma_{u}(T)
\end{array}$$
where $\gamma_{u} :[0,T]\rightarrow M$ is the unique solution to the Cauchy problem:
\begin{equation}\label{ODE}
\left\{ 
\begin{array}{lcl}
\dot{\gamma}_{u}(t) &=&\displaystyle{ \sum_{i=1}^{k} u_{i}(t) X^{i}(\gamma_{u}(t))},\ \forall t\in [0,1]\\
\gamma_{u}(0)&=& x
\end{array}
\right..
\end{equation}
A control $u$ is called singular if and only if it is a critical point of $E^{x}$, and regular if not. An horizontal path $\gamma$ is said to be singular (resp. regular) if and only if any control $u$ associated to $\gamma$ (i.e. $\gamma=\gamma_{u}$ solution of (\ref{ODE})) is singular (resp. regular) for $E^{x}$.\\ 
For every $x\in M$ and every $T>0$, we denote by $\Omega^{R}_{x, T}$ the set of regular minimizing geodesics $\gamma:[0,T]\rightarrow M$ starting at $x$.
We also denote by $\Omega^{S}_{x, T}$ the set of singular minimizing geodesics $\gamma:[0,T]\rightarrow M$ starting at $x$. \\

 The notion of singular curves play a major role in this paper. In absence of singular minimizing geodesics, sub-Riemannian distances enjoy the same kind of regularity as Riemannian distances at least outside the diagonal. We recall that the diagonal of $M\times M$ is the set of all pairs of the form $(x, x)$ with $x\in M$. Following previous results by Ambrosio-Rigot~\cite{AR04} and Agrachev-Lee~\cite{AL09}, Figalli and Rifford (see~\cite{FR10}) proved that local lipschitzness of the sub-Riemannian distance outside the diagonal is sufficient to guarantee existence and uniqueness of optimal transport maps (see also the textbook~\cite{Rif14} by Rifford).\\
 
 In general, we do not know if the Monge problem (for the sub-Riemannian quadratic cost) admits solutions if there are singular minimizing curves. For a two-rank distribution $ \Delta$ on a three-dimensional manifold $M$, we have existence and uniqueness of optimal transport maps for the sub-Riemannian quadratic cost because non-trivial singular horizontal paths are included in the Martinet surface $\Sigma_{\Delta}$ given by $\Sigma_{\Delta}:=\{ x\in M | \ \Delta(x) + [\Delta, \Delta](x) \neq T_{x}M \}$ which has Lebesgue measure zero. The first relevant case to consider is the one of rank-two distributions in dimension four. In this case, as shown by Sussman~\cite{Sus96}, singular horizontal paths can be seen (locally) as the orbits of a smooth vector field, at least, outside a set of Lebesgue measure zero.\\

The definition of a real analytic manifold is similar to that of a smooth manifold. We begin by recalling that an analytic function $f$ is an infinitely differentiable function such that the Taylor series at any point $x_{0}$ in its domain, converges to $f(x)$ for $x$ in a neighborhood of $x_{0}$. We say that a manifold $M$ of dimension $n$ is real analytic if transition maps are analytic. We provide $M$ with a real analytic distribution $\Delta$ of rank $m$ $(m<n)$, that is for each $x\in M$, there is an open neighborhood $\mathcal{U}$ containing $x$ and $m$ analytic vector fields $X^{1},\dots,X^{m}$ on $\mathcal{U}$ such that $$\Delta(y)= Span\{X^{1}(y),\dots,X^{m}(y) \}, \ \forall y\in \mathcal{U}.$$
In this case, the Cauchy problem given in (\ref{ODE}), has a real analytic solution on $M$ for $t\in [0,T]$ and some $T>0$.\\

The aim of this paper is to show that, in the case of rank-two analytic distribution in dimension four, we have existence and uniqueness of optimal transport maps for the sub-Riemannian quadratic cost, as soon as the distribution satisfies some growth condition. \\

We recall that the support of a measure $\mu$, denoted by $supp(\mu)$, refers to the smallest closed set $F\subset M$ of full mass $\mu(F) =\mu(M)=1$.\\
\newpage
Our main result is the following:
\begin{thm}\label{1}
Let $M$ be a real analytic manifold of dimension $4$ and $(\Delta, g)$ be a complete analytic sub-Riemannian structure of rank $2$ on $M$ such that
\begin{eqnarray}\label{ASS}
\forall x\in M,\ \Delta(x) + [\Delta, \Delta](x) \ \text{has dimension 3},
\end{eqnarray}
where $$ [\Delta, \Delta]:=\{ [X,Y]\ |\ X,Y\ \text{sections of}\ \Delta \}.$$
 Let $\mu$, $\nu$ be two probability measures with compact support on $M$ such that $\mu$ is absolutely continuous with respect to the Lebesgue measure $\mathcal{L}^{4}$.\\
 Then, there is existence and uniqueness of an optimal transport map from $\mu$ to $\nu$ for the sub-Riemmannian quadratic cost $c: M\times M \rightarrow [0, + \infty[$ defined by:
$$c(x,y):= d_{SR}^{2}(x,y),\ \forall (x,y) \in M\times M.$$
\end{thm}

\vspace*{1cm}

Our strategy to prove Theorem \ref{1} is twofold. It combines the technique used by Figalli-Rifford~\cite{FR10}  (see also the paper by Agrachev-Lee~\cite{AL09}) which is based on the regularity of the distance function outside the diagonal in absence of singular minimizing curves, together with a localized contraction property for singular curves in the spirit of a previous work by Cavalletti and Huesmann~\cite{CH15}. \\

The paper is organized as follows. In Section 1, we give more details on the strategy of proof. Then Section 2-3 are devoted to prove some required results to achieve existence and uniqueness of optimal transport maps. In Section 4, we finalize the proof of Theorem \ref{1}.

\section{Strategy of proof}\label{strategy}

From now on, we assume that the manifold $M$ has dimension $4$ and is equipped with a complete sub-Riemannian structure $(\Delta, g)$ of rank $2$ such that
$$
\forall x\in M,\ \Delta(x) + [\Delta, \Delta](x) \ \text{has dimension 3}.
$$ 
We fix $\mu, \nu$ two probability measures compactly supported on $M$ such that $\mu$ is absolutely continuous with respect to the Lebesgue measure. As it is well-know (see~\cite{Vil08}), since $c=d_{SR}^2$ is continuous on $M\times M$, the Kantorovitch transport problem between $\mu$ and $\nu$ with cost $c$ admits at least one solution and there is a pair of Kantorovitch potentials $(\varphi, \varphi^{c})$ solution of the dual problem satisfying the equations (\ref{kanto1}) and (\ref{kanto2}). Moreover, we denote by $\Gamma$ the contact set of the pair $(\varphi, \varphi^{c})$ given by
$$
\Gamma:=\Bigl \{ (x,y) \in M\times M |\ \varphi^{c}(y) - \varphi(x) =c(x,y)\Bigr \}.
$$
We get that (see Corollary 3.2.14 \cite{Rif14}):
\begin{center} a transport plan $\alpha \in \Pi(\mu, \nu)$ is optimal if and only if $\alpha( \Gamma)=1$.\end{center}
 In other words, the problem of existence and uniqueness of optimal transport maps can be reduced to prove that $\Gamma$ is concentrated on a graph, that is to show that for $\mu$--almost every point $x \in M$ the set 
$$
 \Gamma(x) :=\Bigl \{ y\in M |\ (x,y) \in \Gamma \Bigr \}\ \text{is a singleton}.
 $$
\vspace*{0.5cm} 
 
Following ~\cite{FR10}, let us  introduce the following definition:

\begin{defi}\label{2}
We call "moving" set $\mathcal{M}$ and "static" set $\mathcal{S}$ respectively the sets defined as follows:

$$\mathcal{S}:=\Bigl \{ x\in M | \ x\in \Gamma(x)\Bigr \},$$

$$\mathcal{M}:=\Bigl \{ x\in M | \ x\notin \Gamma(x)\Bigr \}.$$

\end{defi}

We note that $\mathcal{M}$ is an open subset of $M$. In fact, we can easily check that $\mathcal{M}$ coincides with the set 
$$\{ x\in M\ |\ \varphi(x)\neq \varphi^{c}(x)\} = \{ x\in M\ |\ \varphi(x)> \varphi^{c}(x)\},$$
 which is open by continuity of $\varphi$ and $\varphi^{c}$. \\
 
Since both $supp(\mu)$ and $supp(\nu)$ are compact and the metric space $(M, d_{SR})$ is complete, there are $x_{0}\in M$ and a constant $L>0$ such that 
$$\displaystyle{supp(\mu) \bigcup supp(\nu) \subset B_{SR}(x_{0}, L/4)}$$
where $\displaystyle{B_{SR}(x_{0}, L/4)}$ is the open ball in $\R^{4}$ centered at $x_{0}$ of radius $L$.\\
As a consequence, any minimizing geodesic $\gamma:[0,1]\rightarrow M$ from $x\in supp(\mu)$ to $y\in supp(\nu)$ is contained in $\displaystyle{B_{SR}(x_{0}, L/2)}$.\\

From now on, we work in the compact set $\displaystyle{B_{SR}(x_{0}, L/2)}$ of diameter $L$ and so, we proceed as if $M$ were a compact manifold.

\begin{center}
\begin{tikzpicture}[scale=0.5]
\draw [very thick](5,0) ellipse (12cm and 6cm);
\draw (-6,3) node[left]{$B_{SR}(x_{0}, L/2)$};
\draw[->,>=latex] (0,2) to[bend left] (10,2);
\draw[->,>=latex] (0,1) to[bend left] (10,1);
\draw[->,>=latex] (0,0) to[bend left] (10,0);
\draw[->,>=latex] (0,-1) to[bend left] (10,-1);
\draw[->,>=latex] (0,-2) to[bend left] (10,-2);
\draw [thick] plot [smooth cycle] coordinates {(2,4)(-4, 0) (1,-4)(3.8,0) };
\draw (0,-4.3) node [below]{supp $\mu$};
\draw [thick] plot [smooth cycle] coordinates {(8,4)(7,2) (8,0)(7,-1.5)(10,-4)(12,-3)(14,3.2) };
\draw (10,-4.3) node [below]{supp $\nu$};
\end{tikzpicture}
\end{center}
\vspace*{1cm}

As in \cite{FR10}, we shall show that "static" points do not move, i.e. almost every $x\in \mathcal{S}$ is transported to itself. For sake of completeness, the proof of Lemma \ref{PROPstatic} is given in Appendix \ref{A}.

\begin{lem}\label{PROPstatic}
For $\mu-a.e.\ x\in \mathcal{S}$ , we have $\Gamma(x)=\{x\}$.
\end{lem}
\vspace*{0.5cm} 

We need now to show that almost every moving point is sent to a singleton. To this aim, we need to distinguish between two types of moving points. 
\begin{defi}\label{def3}
Let $T>0$. For every $x\in \mathcal{M}$, we set
$$
\Gamma^{S}(x):= \Bigl \{ y \in \Gamma(x) \, \vert \, \exists \gamma\in \Omega^{S}_{x, T}, \gamma(T)= y\Bigr \}
$$
and
$$ 
\Gamma^{R}(x):= \Bigl \{ y \in \Gamma(x) \, \vert \, \exists \gamma\in \Omega^{R}_{x, T}, \gamma(T)= y \Bigr \}.
$$  
Moreover, we let
$$
\mathcal{M}^{S}:= \left\{ x\in \mathcal{M}; \ \Gamma^{S}(x)\neq \emptyset \right\} \quad \mbox{and} \quad \mathcal{M}^{R}:= \left\{ x\in \mathcal{M}; \ \Gamma^{R}(x)\neq \emptyset \right\}.
$$
\end{defi}
\vspace*{0.5cm} 

Note that, by construction, for every $x\in \mathcal{M}$, $\Gamma(x) = \Gamma^{R}(x)\cup \Gamma^{S}(x)$. Furthermore, if there are no non-trivial singular minimizing curves then $\mathcal{M}^{S}= \emptyset$.
\vspace*{0.3cm}

First, using techniques reminiscent to the previous works by Agrachev-Lee~\cite{AL09} and Figalli-Rifford~\cite{FR10}, we prove that
\begin{pro}\label{2}
For $\mathcal{L}^{4}$-a.e. $x\in \mathcal{M}^{R}$, $\Gamma^{R}(x)$ is a singleton.
\end{pro}
\vspace*{0.3cm}

Then, using a localized contraction property for singular curves which holds thanks to  (\ref{ASS}), the technique developed by  Cavalletti and Huesmann~\cite{CH15} allows to show that 
\begin{pro}\label{3}
For $\mathcal{L}^{4}$-a.e. $x\in \mathcal{M}^{S}$, $\Gamma^{S}(x)$ is a singleton.
\end{pro}

It remains to show that for almost every $x\in M$, $\Gamma (x)$ is a singleton. Again this will follow from a local contraction property together with the approach of Cavalletti and Huesmann~\cite{CH15}, see Section \ref{SECend}.

\section{Proof of Proposition \ref{1}}\label{SECregular}
Argue by contradiction, by assuming that there is a compact set $A\subset \mathcal{M}^{R}$ of positive Lebesgue measure such that
\begin{equation}\label{assumption}\forall x\in A,\ \Gamma^{R}(x)\ \text{is not a singleton.}\end{equation}

We may assume that $A$ is contained in a  chart $(\mathcal{V}, \phi_{\mathcal{V}})$ of $M$. Without loss of generality, we may assume that $\mathcal{V}$ is an open subset of $\R^{4}$ where we can use the local set of coordinates $(x_{1}, x_{2}, x_{3}, x_{4})$.\\

\vspace*{0.7cm}

For every $k \in \N$, we define the set 
\begin{multline}
W_{k}:=\Bigl\{ x\in \mathcal{M} \, \vert \, \exists p_{x}\in \R^{4}; |p_{x}|\leq k\ \text{and}\\ \varphi(x) \leq \varphi(z)- <p_{x}, x-z> + k\ |x-z|^{2}, \forall z \in \bar{B}(x, 1/k) \Bigr\},
\end{multline}
where $\bar{B}(x,1/k)$ denotes the closed ball in $\R^{4}$ centered at $x$ with radius $1/k$.\\
The set $W_{k}$ is well-defined, up to a change of coordinates, for $k$ large enough.
\vspace*{0.5cm}

\begin{lem}\label{4}
$\displaystyle{\mathcal{M}^{R} \subset \bigcup_{k\in \N} W_{k}}.$
\end{lem}

\begin{proof}[\textbf{Proof of Lemma \ref{4}}]
Let $\bar{x} \in \mathcal{M}^{R}$, then there are $\bar{y}\in \Gamma^{R}(\bar{x})$ and\\ $\bar{\gamma}: [0, 1]\rightarrow M$ a regular horizontal path steering $\bar{y}$ to $\bar{x}$. There exist an open neighborhood $\mathcal{V}$ of $\bar{\gamma}([0,1])$ and an orthonormal family (with respect to $g$) $\mathcal{F}$ of two  vector fields $X^{1}, X^{2}$ such that 

$$ \Delta(z)= Span\Bigl \{ X^{1}(z), X^{2}(z)\Bigr \}, \ \forall z \in \mathcal{V}.$$

According to a change of coordinates if necessary, we can assume that $\mathcal{V}$ is an open subset of $\R^{4}$. Moreover, there is a control $ \bar{u} \in L^{2}([0,1], \R^{2})$ associated to $\bar{\gamma}$, ie.
$$
\dot{\bar{\gamma}}(t)= \sum_{i=1}^{2} \bar{u}_{i}(t) X^{i}(\bar{\gamma} (t)), \ \forall t \in [0,1].
$$
We recall that the set of minimizing geodesics between $\bar{x}$ and $\bar{y}$ is compact with respect to the uniform topology: if $(y_{k})_{k}$ is a sequence converging uniformly to $y$ then, the sequence $(\gamma_{k})_{k}$ of minimizing geodesics joining $x$ to $y_{k}$ converges uniformly to $\bar{\gamma}$ and the sequence $(u_{k})_{k}$ of controls associated to $(\gamma_{k})_{k}$ converges uniformly to $\bar{u}$ in $ L^{2}([0,1], \R^{m})$. Then, there exists an open neighborhood $\mathcal{O}_{\bar{x}}$ of $\bar{x}$ such that $\forall z \in \mathcal{O}_{\bar{x}}$, every minimizing geodesic joining $\bar{y}$ to $z$ is contained in $\mathcal{V}$.\\
Since $\bar{\gamma}$ is regular, there exist $v^{1}, v^{2}, v^{3}, v^{4} \in L^{2}([0,1], \R^{2})$ such that the linear operator \begin{equation}\label{*}
\begin{array}{lcl}
\R^{4}& \rightarrow & \R^{4}\\
\alpha & \mapsto & \displaystyle{\sum_{i=1}^{4} \alpha_{i} D_{\bar{u}}E^{\bar{y}}(v^{i})}
\end{array}
\end{equation}
is invertible.\\
Recall that $C^{\infty}([0,1],\R^{2})$ is dense in $L^{2}([0,1], \R^{2})$, we can assume that we have $v^{1}, v^{2}, v^{3}, v^{4} $ in $C^{\infty}([0,1],\R^{2})$.

Define locally $$\begin{array}{lccl}
\mathcal{F}:&\R^{4}& \rightarrow & \R^{4}\\
&\alpha &\mapsto &  \displaystyle{E^{\bar{y}}(\bar{u}+ \sum_{i=1}^{4} \alpha_{i} v^{i})}
\end{array}.$$\\
This mapping is well-defined and of class $C^{2}$ in the neighborhood of zero. It satisfies $\mathcal{F}(0)= \bar{x}$ and its differential at $0$ is invertible. \\
By the Local Inverse Function Theorem, there exist an open ball $\mathcal{B}$ of $\R^{4}$ centered at $\bar{x}$ and a function $\mathcal{G}: \mathcal{B}\rightarrow \R^{4}$ of class $C^{2}$ such that 

$$\mathcal{F}\circ \mathcal{G} (z)= z, \ \forall z \in \mathcal{B}.$$

 $$\forall z \in \mathcal{B},\ d^{2}_{SR}(z, \bar{y}) \leq || \bar{u}+ \sum_{i=1}^{n}(\mathcal{G}(z))_{i}v^{i} ||^{2}_{L^{2}}.$$
 \vspace*{0.5cm}
 
\begin{center}
\begin{tikzpicture}[scale=1]
\draw [thick] plot [smooth cycle] coordinates {(-2,-2)(-4,0)(-2,4)(4,3)(8,4) (10,3)(11,0)(9,-2)(5,-2.5) };
\draw (-3.5,3) node[below]{$\mathcal{V}$};
\draw (0,1) node {$\bullet$};
\draw (-0.5,1) node[above]{$\bar{x}$};
\draw (8.1,1) node {$\bullet$};
\draw (8.2,1) node[right]{$\bar{y}$};
\draw (0, 1) .. controls (1,0.5) and (8, 0.5) .. (8.1, 1);
\draw (5,0.7) node[above]{$\bar{\gamma}\leftrightarrow\bar{u}$};
\draw (-0.4,1.2) node[right]{$=E^{\bar{y}}(\bar{u})$};
\draw[thick, dashed] (0, 1) circle (1.5cm);
\draw (-1,3) node[below]{$\mathcal{B}$};
\draw (0.6,0) node {$\bullet$};
\draw (0.6,0) node[left]{$z$};
\draw [dashed](0.6, 0) .. controls (1.8,-0.5) and (8, 0) .. (8.1, 1);
\draw (5,-0.1) node[below]{$\displaystyle{\bar{u}+ \sum_{i=1}^{n}(\mathcal{G}(z))_{i}v^{i}}$};
\end{tikzpicture}
\end{center} 
\vspace*{0.5cm}

Define $\phi^{\bar{x},\bar{y}}(z):= ||\displaystyle{ \bar{u}+ \sum_{i=1}^{n}(\mathcal{G}(z))_{i}v^{i} }||^{2}_{L^{2}}$, $\forall z\in \mathcal{B}$. Then, we conclude that there is a $C^{2}$ function $\phi^{\bar{x},\bar{y}}: \mathcal{B}\rightarrow \R^{4}$ such that
$$\phi^{\bar{x},\bar{y}}(z) \geq d^{2}_{SR}(z, \bar{y}),\ \forall z\in \mathcal{B}\ \text{and}\ \phi^{\bar{x},\bar{y}}(\bar{x})=d^{2}_{SR}(\bar{x},\bar{y}).$$

Recall that, by the definition of the Kantorovitch potentials,  for every $z\in M$, we have 
$$\left \{\begin{array}{lcl}
\varphi (z) &\geq& \varphi^{c}(\bar{y}) -d^{2}_{SR}(z, \bar{y})\\
\varphi (\bar{x}) &= &\varphi^{c}(\bar{y}) -d^{2}_{SR}(\bar{x}, \bar{y})
\end{array}\right. .$$
Then, $\forall z\in \mathcal{B}$, $$\left \{\begin{array}{lcl}
\varphi (z) &\geq& \varphi^{c}(\bar{y}) -\phi^{\bar{x},\bar{y}}(z)\\
\varphi (\bar{x}) &= &\varphi^{c}(\bar{y}) -\phi^{\bar{x}, \bar{y}}(\bar{x})
\end{array}\right. .$$
Define $\psi^{\bar{x},\bar{y}}(z):=  \varphi^{c}(\bar{y}) -\phi^{\bar{x},\bar{y}}(z), \forall z\in \mathcal{B}$. Hence, we put locally a $C^{2}$ function under the graph of $\varphi$ with a uniform control on the $C^{2}$ norm of $\psi^{\bar{x},\bar{y}}$. Then, for $\bar{x}\in \mathcal{M}^{R}$, we can find $k\in \N$ such that there is $p_{\bar{x}}\in \R^{4}$ with $|p_{\bar{x}}|\leq k$ verifying
$$ \varphi(\bar{x}) \leq \varphi(y) - <p_{\bar{x}}, \bar{x}-y > + k\ |\bar{x}-y|^{2}, \ \forall y\in B(\bar{x},1/k).$$
\end{proof}
\vspace*{0.7cm}

We are ready to complete the proof of Proposition \ref{2}.\\

Since $\displaystyle{\mathcal{M}^{R}\subset \bigcup_{k\in \N} W_{k}}$ (by Lemma \ref{3}), there exists $k\in \N$ such that
$$A_{k}:= A\cap W_{k}\ \text{is of positive Lebesgue measure}.$$
Let $\bar{x}$ be a density point of $A_{k}$ and $\bar{y}\in \Gamma^{R}(\bar{x})$. By the definition of the Kantorovitch potentials, we have that 
$$\varphi(\bar{x}) + d_{SR}(\bar{x}, \bar{y})^{2} \leq \varphi(z) + d_{SR}(z, \bar{y})^{2}, \forall z\in M$$
$$\Rightarrow \varphi(\bar{x}) + d_{SR}(\bar{x}, \bar{y})^{2}- \varphi(z) \leq  d_{SR}(z, \bar{y})^{2}, \forall z\in M.$$
\vspace*{0.5cm}

We define the function $\begin{array}{lccl}
\rho^{\bar{x}}:&M& \rightarrow &\R\\
&z&\mapsto & \rho^{\bar{x}}(z):= \varphi(\bar{x}) + d_{SR}^{2}(\bar{x}, \bar{y})- \varphi(z)
\end{array}$ verifying 
\begin{equation}\label{s1}\rho^{\bar{x}}(z) \leq d_{SR}^{2}(z, \bar{y}), \forall z\in M \ \text{and equality for}\ z=\bar{x}.\end{equation}
\vspace*{0.3cm}

Let $\displaystyle{\tilde{A_{k}}:= A_{k}\cap B(\bar{x}, 1/2k)}$. For every $y\in \tilde{A}_{k}$, there is $p_{y}\in \R^{4}$, $|p_{y}|\leq k$ such that
$$\varphi(y)\leq \varphi(z) - <p_{y}, y-z> +k\ |y-z|^{2},\ \forall z\in B(y, 1/k).$$
\vspace*{0.3cm}

We define the function $\tilde{\varphi}: B(\bar{x},1/2k) \rightarrow \R$ as follows

$$\tilde{\varphi}(x) = \displaystyle{\sup_{y\in \tilde{A}_{k}} \Psi_{y}(x)}, \ \forall y\in B(\bar{x}, 1/2k)$$

where $$ \forall y\in \tilde{A}_{k},\ \Psi_{y}(x):= \varphi(y) + <p_{y}, y-x> - k\ |y-x|^{2}.$$
\vspace*{0.4cm}

We claim that for every $x\in \tilde{A}_{k}$, $\tilde{\varphi}(x)= \varphi(x)$. Let us prove our claim.\\
In fact, for every $x\in \tilde{A}_{k}$,we have 

$$\tilde{\varphi}(x) \geq \Psi_{y}(x),\ \forall y\in \tilde{A}_{k},$$

that is 

$$\tilde{\varphi}(x) \geq  \varphi(y) + <p_{y}, y-x> - k\ |y-x|^{2},\ \forall y\in \tilde{A}_{k} .$$

In particular, for $y=x\in\tilde{A}_{k}$, we obtain $$\varphi(x) \leq \tilde{\varphi}(x).$$

Assume that there is $x\in \tilde{A}_{k}$ such that $\varphi(x) < \tilde{ \varphi}(x)$.\\
Then, there is $y\in \tilde{A}_{k}$, $y\neq x$ such that 

$$ \varphi(x) < \ \Psi_{y}(x)$$

that is

\begin{equation}\label{i1}\varphi(x) <\  \varphi(y) + <p_{y}, y-x> - k\ |y-x|^{2} .\end{equation}\\
Or, $x, y\in \tilde{A}_{k}$, then $x\in B(y, 1/k)$. \\
So,
 $$\varphi(y) \leq \varphi(x) - <p_{y}, y-x> + k |x-y|^{2}$$

$$\Rightarrow  \varphi(y)+ <p_{y}, y-x>  - k\ |x-y|^{2}\leq  \varphi(x)$$

which contradicts inequality (\ref{i1}). And the conclusion follows.
\vspace*{0.7cm}

Moreover, let $y\in \tilde{A}_{k}$ be fixed. There exists a neighborhood $B(y,1/k)$ of $y$ contained in $B(\bar{x},1/2k)$  
such that for every $x\in B(y,1/k)$, there is $\tilde{p}_{x}\in \R^{4}$ such that $\forall x' \in B(y, 1/k)$, we have

$$\begin{array}{lcl}
\vspace*{0.3cm}

\Psi_{y}(x)-\Psi_{y}(x')&=& <p_{y}, x'-x> + k (|x'-y|^{2}- |x-y|^{2})\\
\vspace*{0.3cm}

&\leq& <p_{y}, x'-x> + k |x'-x|^{2} - 2k <y-x, x'-x>\\
\vspace*{0.3cm}

&\leq& <p_{y}-2k(y-x), x'-x> + k |x'-x|^{2} \\
\end{array}$$

Take $\tilde{p}_{x}:= p_{y}-2k(y-x)$, we obtain 
$$ \Psi_{y}(x)\leq \Psi_{y}(x') - <p_{y}-2k(y-x), x'-x> + k |x'-x|^{2}.$$

\vspace*{0.3cm}

This means that for every $y\in \tilde{A}_{k}$, $\Psi_{y}$ is locally semiconvex on $B(\bar{x}, 1/2k)$. According to Lemma \ref{lemap2} in Appendix B, since $\tilde{\varphi}$ is the supremum of local semiconvex functions $\Psi_{y}$ among all $y\in \tilde{A}_{k}$, then $\tilde{\varphi}$ is locally semiconvex on $B(\bar{x}, 1/2k)$. By the Rademacher Theorem, $\tilde{\varphi}$ is differentiable almost everywhere on $B(\bar{x}, 1/2k)$.\\

We also define the function 
$$\begin{array}{lccl}\tilde{\rho}^{\bar{x}}:&B(\bar{x}, 1/2k)& \rightarrow &\R\\ &z&\mapsto & \tilde{\rho}^{\bar{x}}(z):= \tilde{\varphi}(\bar{x}) + d_{SR}^{2}(\bar{x},\bar{y})- \tilde{\varphi}(z) \end{array}$$

 such that 
 \begin{equation}\label{s2}\tilde{\rho}^{\bar{x}}= \rho^{\bar{x}} \ \text{on}\ \tilde{A}_{k}.\end{equation}

Here, $\bar{x}$ is fixed and $\tilde{\rho}^{\bar{x}}$ is a function of $z$. By the definition of $\tilde{\rho}^{\bar{x}}$, as $\tilde{\varphi}$ is differentiable at almost every $z\in B(\bar{x}, 1/2k)$, $\tilde{\rho}^{\bar{x}}$ is also differentiable almost everywhere on $B(\bar{x}, 1/2k)$.
\vspace*{0.7cm}

On the other hand, following the proof of Lemma \ref{4}, for $\bar{x}\in \mathcal{M}^{R}$ and $\bar{y} \in \Gamma^{R}(\bar{x})$, there are an open set $\mathcal{B}_{\bar{x}}$ in $\R^{4}$ containing $\bar{x}$ and a $C^{2}$ function $\phi^{\bar{x},\bar{y}}: \mathcal{B}_{\bar{x}} \rightarrow \R$ such that

\begin{equation}\label{s3}\phi^{\bar{x},\bar{y}}(z) \geq d_{SR}^{2}(z,\bar{y}), \forall z\in \mathcal{B}_{\bar{x}}\ \text{and equality for}\ z=\bar{x}.\end{equation}
\vspace*{0.1cm}

Consequently, by (\ref{s1}), (\ref{s2}), (\ref{s3}), we obtain 

$$\tilde{\rho}^{\bar{x}}(z) \leq d_{SR}^{2}(z,\bar{y}) \leq \phi^{\bar{x},\bar{y}}(z),\ \forall z\in \mathcal{B}_{\bar{x}}\cap \tilde{A}_{k}$$ and 
$$\text{equality for}\ z=\bar{x}.$$
\vspace*{0.1cm}

Note that $\phi^{\bar{x},\bar{y}}$ is a $C^{2}$ function and $\tilde{\rho}^{\bar{x}}$ is differentiable almost everywhere on $B(\bar{x}, 1/2k)$. Then,
$$d_{\bar{x}}\phi^{\bar{x},\bar{y}}= d_{\bar{x}}\tilde{\rho}^{\bar{x}}.$$
\vspace*{0.1cm}

It means that there is a unique $\bar{y}\in \Gamma^{R}(\bar{x})$ such that
$$\bar{y} = exp_{\bar{x}}(d_{\bar{x}}\tilde{\rho}^{\bar{x}}) = exp_{\bar{x}}(-d_{\bar{x}}\tilde{\varphi}),$$
with $exp_{\bar{x}}: T^{*}_{\bar{x}}M \rightarrow M$ the sub-Riemannian exponential map from $\bar{x}$. This contradicts assumption (\ref{assumption}) and the conclusion follows.
\vspace*{0.5cm}

\begin{rmk}
The above argument can be used to prove the required result in the general case, with $M$ a smooth connected manifold of dimension $n$ equipped with a complete sub-Riemannian structure $(\Delta, g)$ of rank $m (m<n)$.
\end{rmk}

\section{Proof of Proposition \ref{3}}\label{SECsingular}

Our aim is to prove that
\begin{center}for almost every $x\in \mathcal{M}^{S}$, $\Gamma^{S}(x)$ is a singleton.\end{center}

First, we need to construct a line field, defined on a set of full Lebesgue measure, whose orbits correspond to the singular curves. 
\vspace*{0.7cm}

The following holds (see ~\cite{Sus96}, ~\cite{Rif14},  ~\cite{LS95}) :
\begin{lem}\label{5}
There is an open set $\mathcal{H}$ of full Lebesgue measure on $M$ such that:\\
\begin{equation}\label{EQ2}
\forall x\in \mathcal{H}, \ T_{x}M =  \Delta(x) + [\Delta, \Delta](x) + [\Delta, [\Delta, \Delta]](x).
\end{equation}
\end{lem}

\begin{proof}[\textbf{Proof of Lemma \ref{5}}]

We denote by $\mathscr{S}$ the set given by 

$$\mathscr{S} =\Bigl \{ x\in M | \Delta(x) + [\Delta, \Delta](x) +[\Delta(x),[\Delta, \Delta]](x) \neq T_{x}M \Bigr \}.$$

It is clear that $\mathscr{S}$ is a closed set on $M$ such that condition (\ref{EQ2}) is verified on its complementary set. Let us prove that $\mathscr{S} $ is of Lebesgue measure zero on $M$. For sake of simplicity, we will work locally. In other terms, given $\bar{x}\in M$, there are a local set of coordinates $(x_{1},x_{2},x_{3},x_{4})$  in an open neighborhood $\mathcal{V}$ of $\bar{x}$ and two vector fields $X^{1}, X^{2}$ linearly independent on $\mathcal{V}$ such that

 $$\Delta(x)= Span\Bigl \{ X^{1}(x), X^{2}(x)\Bigr \},\ \forall x\in \mathcal{V}.$$
 
By hypothesis (\ref{ASS}) in Theorem \ref{1}, we have

$$\forall x\in M, \Delta(x)+ [\Delta, \Delta](x) \ \text{has dimension}\ 3.$$
As a consequence, $\Delta + [\Delta, \Delta]$ is a totally nonholonomic distribution of rank $3$ in dimension $4$ with 
$$\Delta(x)+ [\Delta,\Delta](x) = Span\{X^{1}(x), X^{2}(x), X^{3}(x)\},\ \forall x\in \mathcal{V}$$
where $X^{3}= [X^{1}, X^{2}].$\\

According to a change of coordinates  if necessary, we can assume that 
$$ X^{i} =\displaystyle{ \frac{\partial}{\partial x_{i}} + \alpha_{i}(x) \frac{\partial }{\partial x_{4}} }, \forall i=1, 2, 3$$
where $\alpha_{i}: \mathcal{V} \rightarrow \R$, $\forall i=1, 2, 3$ are analytic functions.\\

Hence, $\forall i,j \in \{ 1, 2, 3\}$, we have 
$$\Bigl [X^{i}, X^{j}\Bigr ] = \displaystyle{\Bigl ((\frac{\partial \alpha_{j}}{\partial x_{i}} - \frac{\partial \alpha_{i}}{\partial x_{j}}) + (\frac{\partial \alpha_{j}}{\partial x_{4}} \alpha_{i} - \frac{\partial \alpha_{i}}{\partial x_{4}}\alpha_{j})\Bigr ) \frac{\partial}{\partial x_{4}} },
$$
and
$$ \mathscr{S}:=\Bigl \{ x\in \mathcal{V} |\ \Bigl (\frac{\partial \alpha_{j}}{\partial x_{i}} - \frac{\partial \alpha_{i}}{\partial x_{j}}\Bigr ) + \Bigl (\frac{\partial \alpha_{j}}{\partial x_{4}} \alpha_{i} - \frac{\partial \alpha_{i}}{\partial x_{4}}\alpha_{j}\Bigr )=0, \forall i,j =1, 2, 3 \Bigr \}.
$$
For every $ I=(i_{1},\dots, i_{k}) \in \{ 1, 2, 3\}$, we denote by $X^{I}$ the smooth vector field constructed by the Lie brackets of $X^{1}, X^{2}, X^{3}$ as follows
$$ X^{I} = \Bigl [X^{i_{1}}, [X^{i_{2}},\dots, [X^{i_{k-1}},X^{i_{k}}]\dots ] \Bigr ]. $$
Note  $length(I)$ the length of the Lie brackets $X^{I}$. Since $\Delta+ [\Delta, \Delta]$ is totally nonholonomic distribution, there exists a positive integer $r$ such that 
$$ T_{x}M= Span \Bigl \{ X^{I}(x) | length(I) \leq r\Bigr \}, \forall x\in \mathcal{V} .$$
For every $I$ of $length(I)\geq 2$, there exists a function $g_{I}: \mathcal{V}\rightarrow \R$ such that 
$$ X^{I}(x) = g_{I}(x) \frac{\partial }{\partial x_{4}}, \forall x\in \mathcal{V}. $$
We define the following sets
$$ \mathcal{A}_{k}:=\Bigl \{ x\in \mathcal{V} | \ g_{I}(x)=0, \forall I \ s.t. \ length(I) \leq k\Bigr \}.$$ 
We have $\displaystyle{ \Sigma_{\delta} = \bigcup_{k=2}^{r}\Bigl (\mathcal{A}_{k}\textbackslash \mathcal{A}_{k+1}\Bigr )}$.\\

By the Implicit Function Theorem, each set $\mathcal{A}_{k}\textbackslash \mathcal{A}_{k+1}$ can be covered by a countable union of smooth hypersurfaces. Fix $x\in \mathcal{A}_{k}\textbackslash \mathcal{A}_{k+1}$. \\
There exists some $J=(j_{1}, \dots, j_{k+1})$ of length $k+1$ such that $g_{J}(x)\neq 0$.\\ Put $I=(j_{2}, \dots, j_{k+1})$. Then
  
  $$ g_{J}(x) =\Bigl (\frac{\partial g_{I}}{\partial x_{j_{1}}}(x) +\ \frac{\partial g_{I}}{\partial x_{4}}(x) \alpha_{j_{1}}(x)  \Bigr) \neq 0.$$

Hence, $$\displaystyle{\frac{\partial g_{I}}{\partial x_{j_{1}}}(x) \neq 0}\ \text{or}\ \displaystyle{\frac{\partial g_{I}}{\partial x_{4}}(x) \neq 0}.$$
 We deduce that 
 
$$ \mathcal{A}_{k}\textbackslash \mathcal{A}_{k+1} \subset \displaystyle{\bigcup_{length(I)=k} \Bigl \{x\in \mathcal{V} | \ \exists i\in \{1, \dots, 4 \}\ s.t. \   \frac{\partial g_{I}}{\partial x_{i}}(x) \neq 0\Bigr \}}.$$

It shows that $\mathscr{S}$ is a closed 3-rectifiable set in $M$, so $\mathscr{S}$ is of Lebesgue measure zero on $M$. We can indeed take $\mathcal{H}$ the complementary set of $\mathscr{S}$ in $M$.
\end{proof}
\vspace*{0.5cm}

We need another lemma.
\begin{lem}\label{6}
There exists a line subbundle $L$ of $\Delta$ such that the singular horizontal curves defined on $\mathcal{H}$ are exactly the trajectories described on $L$.\end{lem}

\begin{proof}[\textbf{Proof of Lemma \ref{6}}]

It is sufficient to prove the result in a neighborhood of each point in $\mathcal{H}$. So, let us consider a local frame $\{X^{1}, X^{2} \}$ such that 
$$\Delta(z) = Span\{ X^{1}(z), X^{2}(z)\},\ \forall z\in M.$$ 
Let $\gamma:[0,1]\rightarrow M$ be a trajectory associated to some control $u\in L^{2}([0,1], \R^{2})$. In local coordinates, singular curves can be characterized as follows (see Proposition 1.3.3 \cite{Rif14}):\\

\begin{center} $\gamma$ is singular with respect to $\Delta$ if there is $p: [0,1]\rightarrow (\R^{4})^{*}\textbackslash \{0\}$ satisfying :\end{center}
\begin{equation}\label{CO1}
\dot{p}(t)=\displaystyle{-\sum_{i=1}^{2}u_{i}(t) p(t). D_{\gamma(t)}X^{i},} \ a.e. \ t\in [0,1]
\end{equation}
\begin{equation}\label{CO2}
p(t).X^{i}(\gamma(t))=0, \forall t\in [0,1], \ \forall i=1,2
\end{equation}
\vspace*{0,5cm}

Derivative two times yields for almost every $t\in [0,1]$ such that $u(t)\neq 0$
\begin{equation}\label{CO3}
p(t).\Bigl [X^{1}(t), X^{2}(t)\Bigr](\gamma(t)) =0,
\end{equation} 
and
\begin{equation}\label{CO4}
u_{1}(t) p(t).\Bigl[X^{1},[X^{1}, X^{2}]\Bigr](\gamma(t)) + u_{2}(t) p(t).\Bigl[X^{2},[X^{1}, X^{2}]\Bigr](\gamma(t)) =0 .
\end{equation}
Since $M$ has dimension four and $\Delta+\Bigl[\Delta, \Delta \Bigr]$ has dimension three, there is locally a smooth non-vanishing 1-form $\alpha$ such that 
$$ \alpha_{x}.v =0, \ \forall v\in \Delta(x)+\Bigl [\Delta, \Delta \Bigr](x), \ \forall x\in \mathcal{H}.$$
Then, by (\ref{CO2}), (\ref{CO3})-(\ref{CO4}), we infer that for almost every $t\in [0,1]$ such that $u(t)\neq 0$, we have:
$$
u_{1}(t) \alpha_{\gamma(t)}.\Bigl [X^{1},[X^{1}, X^{2}]\Bigr](\gamma(t)) + u_{2}(t) \alpha_{\gamma(t)}.\Bigl[X^{2},[X^{1}, X^{2}]\Bigr](\gamma(t)) =0 .
$$
By above assumption, for every $x\in \mathcal{H}$, the linear form
$$(\lambda_{1}, \lambda_{2}) \mapsto (\alpha_{x}.\Bigl[X^{1},[X^{1}, X^{2}]\Bigr](x))\lambda_{1}+(\alpha_{x}.\Bigl[X^{2},[X^{1}, X^{2}]\Bigr](x))\lambda_{2}$$
has a kernel of dimension one. This shows that there is a smooth line field (a distribution of rank one) $L\subset \Delta$ on $M$ such that the singular horizontal curves are exactly the integral curves of $L$.
\end{proof}
\vspace*{0.5cm}

We are ready now to prove Proposition \ref{3}. Without loss of generality,  it is sufficient to prove the result locally. We can assume that $(x_{1},x_{2},x_{3},x_{4})$ denotes the coordinates in an open neighborhood $\mathcal{V}$ in $M$ and consider $\{X^{1},X^{2}\}$ a local frame of $\Delta$ such that 
$$\Delta(x) = Span\{X^{1}(x), X^{2}(x) \}, \forall x\in \mathcal{V}.$$

Doing a change of coordinates if necessary, we can assume that 
\begin{eqnarray*} 
X^{1}= \partial_{x_{1}}, & X^{2}=\partial_{x_{2}} + A(.) \partial_{x_{3}} + B(.) \partial_{x_{4}}
\end{eqnarray*}
where $A,B : \mathcal{V}\rightarrow \R$ are smooth functions.
\vspace*{0.5cm}

For the upcoming results, it is important to keep in mind the following notations.
\begin{nota}\label{not}
We denote by $A_{x_{i}}, B_{x_{i}}$ the partial derivative with respect to the variable $x_{i}$, and $A_{x_{i}x_{j}}, B_{x_{i}x_{j}}$ the second partial derivative with respect to the variable $x_{i}$ and  $x_{j}$, of $A$ and $B$ respectively. \\

We compute the Lie brackets of $X^{1}$ and $X^{2}$ :\\

\begin{equation}\label{bra}\hspace*{-6cm}\Bigl [X^{1}, X^{2}\Bigr]= A_{x_{1}} \partial_{x_{3}} + B_{x_{1}} \partial_{x_{4}}\end{equation}

$$ \hspace*{-4.3cm}\Bigl[X^{1}, [X^{1}, X^{2}]\Bigr] = A_{x_{1}x_{1}} \partial_{x_{3}} + B_{x_{1}x_{1}} \partial_{x_{4}}$$

$$\hspace*{-5.5cm}\Bigl[X^{2}, [X^{1}, X^{2}]\Bigr] = E \partial_{x_{3}} + F \partial_{x_{4}}$$
\vspace*{0.2cm}

\hspace*{3.5cm} with $\left\{\begin{array}{lcl} \vspace*{0.5cm}

E &=&A_{x_{2}x_{1}} + A A_{x_{3}x_{1}} + B A_{x_{1}x_{4}} -A_{x_{1}}A_{x_{3}} - B_{x_{1}}A_{x_{4}},\\

 F &=&B_{x_{2}x_{1}} + A B_{x_{3}x_{1}} + B B_{x_{1}x_{4}} -A_{x_{1}}B_{x_{3}} - B_{x_{1}}B_{x_{4}}.\end{array}\right.$ \end{nota}
\vspace*{0.5cm}

By hypothesis (\ref{ASS}) and (\ref{bra}), we can assume that 

\begin{equation}\label{cond}A_{x_{1}}(x)\neq 0, \ \forall x\in \mathcal{V}.\end{equation}
\vspace*{0.5cm}

We denote by $\mathcal{H}^{c}$ the complementary set of $\mathcal{H}$ on $M$ given by
$$\mathcal{H}^{c} =\Bigl \{x\in M |\ \Delta(x) + \Bigl[\Delta, \Delta \Bigr](x) + \Bigl[\Delta, [\Delta, \Delta]\Bigr](x)\neq T_{x}M\Bigr \}.$$

Thus, $\mathcal{H}^{c}$ is a closed set of Lebesgue measure zero on $M$.
\vspace*{1cm}

The above discussion implies indeed the following lemma.
\begin{lem}\label{7}
There exists an analytic horizontal vector field $X$ given by $$X= \alpha_{1} X^{1}+\alpha_{2} X^{2}$$
 with $\alpha_{1}, \alpha_{2}: \mathcal{V}\rightarrow \R$ smooth functions given by
$$\left\{ \begin{array}{lcccl}
\alpha_{1}&=& EB_{x_{1}}&-&FA_{x_{1}}\\
\alpha_{2}&=& B_{x_{1}x_{1}}A_{x_{1}}&-& A_{x_{1}x_{1}}B_{x_{1}}
\end{array}\right.$$
($E$ and $F: \mathcal{V}\rightarrow \R$ smooth functions defined in Notation \ref{not}).\\

The vector field $X$ vanishes on $\mathcal{H}^{c}$ and any solution of the Cauchy problem $\dot{x}(t)= X(x(t))$ is analytic and singular.\\
\end{lem}

\begin{proof}[\textbf{Proof of Lemma \ref{7}}]

Let $T>0$ and let $u\in L^{2}([0,1], \R^{2})$ be a singular control and \\$x:[0,T]\rightarrow M$ be a solution to the Cauchy problem $$\dot{x}(t)= u_{1}(t)X^{1}(x(t)) +  u_{2}(t)X^{2}(x(t)), \  a.e.\ t\in [0,T].$$
There exists an absolutely continuous arc $p:[0,T]\rightarrow (\R^{4})^{*}\textbackslash\{0\}$ such that \\
\begin{equation}
\dot{p}(t) =-u_{1}(t) p(t).D_{x(t)}X^{1} - u_{2}(t) p(t).D_{x(t)}X^{2}, a.e.\ t\in [0, T]
\label{7a}\end{equation}

\begin{equation}
p(t).X^{1}(x(t))= p(t).X^{2}(x(t)) =0 , \forall t\in [0,T]
\label{7b}
\end{equation}
\\
Taking the derivatives in (\ref{7b}) gives \begin{equation}\label{7c}p(t).[X^{1}, X^{2}](x(t)) =0, \ \forall t\in [0,T]\end{equation}
which implies that $\forall t \in [0,T]$,
$$\left \{ \begin{array}{l}
\vspace*{0.3cm}

p_{1}(t) =0 \\
\vspace*{0.3cm}

p_{2}(t) + A(x(t)) p_{3}(t) + B(x(t)) p_{4}(t)=0\\
\vspace*{0.3cm}

A_{x_{1}}(x(t))p_{3}(t) + B_{x_{1}}(x(t))p_{4}(t)=0
\end{array}\right.$$
Assume that condition (\ref{cond}) is true, then we obtain \\
$$\displaystyle{p(t)=(0, [A(x(t))\frac{B_{x_{1}}}{A_{x_{1}}}(x(t))-B(x(t))]p_{4}(t),-\frac{B_{x_{1}}}{A_{x_{1}}}(x(t))p_{4}(t), p_{4}(t))},\ \ \forall t\in [0,T].$$

By taking the derivatives in (\ref{7c}), we obtain for every $t\in [0,T]$
$$u_{1}(t)p(t).[X^{1},[X^{1}, X^{2}]](x(t)) + u_{2}(t)p(t).[X^{2},[X^{1}, X^{2}]](x(t)) =0$$
$$\Rightarrow \displaystyle{u_{1}(t)(p_{3}(t)A_{x_{1}x_{1}}+  p_{4}(t)B_{x_{1}x_{1}}) + u_{2}(t)(p_{3}(t) E +  p_{4}(t)F)=0}.$$
We can write 
$$\left \{ 
\begin{array}{lcccl}
u_{1}(t)&=& -(p_{3}(t) E + p_{4}(t) F)&=& -p_{4}(t)(F - \displaystyle{\frac{B_{x_{1}}}{A_{x_{1}}}}E)\\
u_{2}(t)&=& p_{3}(t) A_{x_{1}x_{1}} + p_{4}(t) B_{x_{1}x_{1}})&=& p_{4}(t)(B_{x_{1}x_{1}}- A_{x_{1}x_{1}}\displaystyle{\frac{B_{x_{1}}}{A_{x_{1}}}})\\
\end{array}.
\right.$$

Assume that $p_{4}(t)=1, \forall t\in [0,1]$, we obtain
\begin{equation}\label{7d}
\left \{\begin{array}{lcl}
\alpha_{1}(x)&=& E B_{x_{1}} - F A_{x_{1}}\\
\alpha_{2}(x)&=& A_{x_{1}}B_{x_{1}x_{1}} - B_{x_{1}} A_{x_{1}x_{1}}\\
\end{array}\right.
\end{equation}
\end{proof}

\begin{lem}\label{7i}
There is a positive constant $C>0$ such that
 $$div_{x}X \geq -C |X(x)|, \ \forall x\in \mathcal{V}.$$
\end{lem}
\begin{proof}[\textbf{Proof of Lemma \ref{7i}}]
Let us compute the divergence of $X$. For every $x\in \mathcal{V}$,\\

$div_{x}X = \alpha_{1}(x) div_{x}X^{1}+ \alpha_{2}(x) div_{x}X^{2} + X^{1}(\alpha_{1}) + X^{2}(\alpha_{2})$\\

\hspace*{1.25cm}$=\displaystyle{  \alpha_{2}(x) div_{x}X^{2} + B_{x_{1}} ( A_{x_{1}x_{2}x_{1}} + A_{x_{1}} A_{x_{3}x_{1}}+ A A_{x_{1}x_{3}x_{1}} + B_{x_{1}} A_{x_{1}x_{4}}}$\\

\hspace*{2cm}$\displaystyle{ + B A_{x_{1}x_{1}x_{4}} -  A_{x_{3}} A_{x_{1}x_{1}}  -  A_{x_{1}} A_{x_{1}x_{3}} - B_{x_{1}x_{1}} A_{x_{4}} - B_{x_{1}}A_{x_{1}x_{4}}   )      }$\\

\hspace*{2cm}$\displaystyle{-A_{x_{1}} ( B_{x_{1}x_{2}x_{1}} + A_{x_{1}} B_{x_{3}x_{1}}+ A B_{x_{1}x_{3}x_{1}} + B_{x_{1}} B_{x_{1}x_{4}}   + B B_{x_{1}x_{1}x_{4}}  }$\\

\hspace*{2cm}$\displaystyle{-  B_{x_{3}} A_{x_{1}x_{1}}  -  A_{x_{1}} B_{x_{1}x_{3}} - B_{x_{1}x_{1}} B_{x_{4}} - B_{x_{1}}B_{x_{1}x_{4}}   ) + E B_{x_{1}x_{1}} }$\\

\hspace*{2cm}$\displaystyle{ -F A_{x_{1}x_{1}} + A_{x_{2}x_{1}} B_{x_{1}x_{1}} + A_{x_{1}}B_{x_{2}x_{1}x_{1}} - B_{x_{2}x_{1}} A_{x_{1}x_{1}} - B_{x_{1}} A_{x_{2}x_{1}x_{1}}   }$\\

\hspace*{2cm}$\displaystyle{ +A A_{x_{3}x_{1}} B_{x_{1}x_{1}} + A A_{x_{1}} B_{x_{3}x_{1}x_{1}} -  A B_{x_{3}x_{1}} A_{x_{1}x_{1}}   - A B_{x_{1}} A_{x_{3}x_{1}x_{1}} }$\\

\hspace*{2cm}$\displaystyle{+ B A_{x_{4}x_{1}} B_{x_{1}x_{1}} + B A_{x_{1}}B_{x_{4}x_{1}x_{1}}  - B B_{x_{4}x_{1}} A_{x_{1}x_{1}}  - B B_{x_{1}}A_{x_{4}x_{1}x_{1}}  } $\\

\hspace*{1.25cm}$=\displaystyle{  \alpha_{2}(x) div_{x}X^{2} + E B_{x_{1}x_{1}} - F A_{x_{1}x_{1}} }$\\

\hspace*{2cm}$\displaystyle{ + B_{x_{1}x_{1}}( B A_{x_{4}x_{1}}+ A A_{x_{3}x_{1}} + A_{x_{2}x_{1}} + A_{x_{1}} B_{x_{4}} - B_{x_{1}} A_{x_{4}}   )  } $\\

\hspace*{2cm}$\displaystyle{ + A_{x_{1}x_{1}}( -B B_{x_{4}x_{1}}- A B_{x_{3}x_{1}} - B_{x_{2}x_{1}} + A_{x_{1}} B_{x_{3}} - B_{x_{1}} A_{x_{3}}   )  } $\\

\hspace*{1.25cm}$=\displaystyle{  \alpha_{2}(x) div_{x}X^{2} + E B_{x_{1}x_{1}} - F A_{x_{1}x_{1}} }$\\

\hspace*{2cm}$\displaystyle{ + B_{x_{1}x_{1}} A_{x_{1}} B_{x_{4}} + B_{x_{1}x_{1}} (E+ A_{x_{1}}A_{x_{3}})- A_{x_{1}x_{1}} B_{x_{1}} A_{x_{3}} -A_{x_{1}x_{1}} (F + B_{x_{1}}B_{x_{4}}) }$\\

\hspace*{1.25cm}$=\displaystyle{  \alpha_{2}(x) div_{x}X^{2} + 2 E B_{x_{1}x_{1}} - 2 F A_{x_{1}x_{1}} }$\\

\hspace*{2cm}$\displaystyle{ + B_{x_{1}x_{1}} (A_{x_{1}} B_{x_{4}} + A_{x_{1}}A_{x_{3}}) - A_{x_{1}x_{1}}(B_{x_{1}} A_{x_{3}} + B_{x_{1}}B_{x_{4}} ) } $\\

\hspace*{1.25cm}$=\displaystyle{  \alpha_{2}(x) div_{x}X^{2} + 2 E B_{x_{1}x_{1}} - 2 F A_{x_{1}x_{1}} }$\\

\hspace*{2cm}$\displaystyle{ + (B_{x_{1}x_{1}} A_{x_{1}} -  A_{x_{1}x_{1}} B_{x_{1}} ) ( A_{x_{3}} + B_{x_{4}} )}$\\

\hspace*{1.25cm}$= 2\ B_{x_{1}x_{1}}E - 2\ A_{x_{1}x_{1}}F + 2\ \alpha_{2}(x) div_{x}X^{2} $.\\

By (\ref{7d}), we can write\ $\displaystyle{B_{x_{1}x_{1}}= \frac{\alpha_{2}+ B_{x_{1}}A_{x_{1}x_{1}}}{A_{x_{1}}}}$ \ and \  $\displaystyle{F= \frac{E B_{x_{1}}-\alpha_{1}}{A_{x_{1}}}}$.\\

Hence, $div_{x}X=\displaystyle{ 2\ \alpha_{2} \frac{E}{A_{x_{1}}} + 2\ \alpha_{1} \frac{A_{x_{1}x_{1}}}{A_{x_{1}}} + 2\ \alpha_{2}div_{x}X^{2} }$\\

\hspace*{2,6cm}$=\displaystyle{ 2\ \alpha_{2}( \frac{E}{A_{x_{1}}}+ div_{x}X^{2}) + 2\ \alpha_{1} \frac{A_{x_{1}x_{1}}}{A_{x_{1}}} }$\\
\vspace*{0,5cm}

As we noticed before, without loss of generality, we proceed as if $M$ is a compact manifold. Then, $\Big(\displaystyle{E/A_{x_{1}}+ div_{x}X^{2}}\Big)$ and $\Big(\displaystyle{A_{x_{1}x_{1}}/A_{x_{1}}}\Big)$ are bounded functions on $M$. There exist $c_{1}, c_{2}>0$ such that 
$$\displaystyle{| \frac{A_{x_{1}x_{1}}}{A_{x_{1}} (x)}|}\leq c_{1} \ \text{and}\ \displaystyle{|\frac{E}{A_{x_{1}}}(x)+ div_{x}X^{2}| \leq c_{2},\ \forall x\in \mathcal{V}}.$$
Thus, $$div_{x}X \geq -k|\alpha_{1}| -k' |\alpha_{2}|,\ \forall x\in \mathcal{V}$$

\hspace*{1.3cm}$$\geq -C|X(x)|, \forall x\in \mathcal{V}$$

with $C=\max \{c_{1}, c_{2} \}>0$ positive constant.
 \end{proof}
\vspace*{0.5cm}

The following process is equivalent to the process introduced by Belotto and Rifford \cite{BR16} to set the contraction property.\\

Let $\varepsilon \in \{\-1, +1\}$ and $T>0$, we denote by $(\varphi^{X}_{\varepsilon t})$ the analytic flow of the vector field $X$ generating locally singular minimizing geodesics.\\

For every subset $A$ in $\mathcal{V}$, we set $$A^{S}_{t}= \varphi^{X}_{\varepsilon t}(A),\ \forall t\in [0,T]\ \text{and}\ A^{S}_{0}=A.$$
We denote by $l(A, t):=\displaystyle{ \sup_{x\in A}\ length \ \varphi^{X}_{\varepsilon t}(A) = \sup_{x\in A} \int_{0}^{t}|X(\varphi^{X}_{\varepsilon s}(x))| \mathrm{d}s},$\\

where $|X(\varphi^{X}_{\varepsilon s}(x))|$ stands form the norm of $X(\varphi^{X}_{\varepsilon s}(x))$ with respect to $g$.\\

We recall that there is $L>0$, already defined in section \ref{strategy}, such that for every $x\in A$, we have
\begin{equation}\label{len}
\int_{0}^{t}|X(\varphi^{X}_{\varepsilon s}(x))| \mathrm{d}s \leq L, \ \forall t\in [0, T].
\end{equation}

We state now divergence formulas, one of the main tool of the present paper (see \cite{BR16}, Proposition B.1).\\

\begin{lem}\label{8}
For every compact $A$ in $M$, there is a smooth function \\ $\mathcal{J}: [0,T]\times A \rightarrow [0, + \infty[$ such that for every $t\in [0, T]$, we have:
\begin{equation}\label{8a}\mathcal{J}(0, z)=1\ \ \text{and}\ \ \displaystyle{\frac{\partial \mathcal{J}}{\partial t}}(t,z) = div\ X(\varphi^{X}_{\varepsilon t}(z))\ \mathcal{J}(t,z)\end{equation}
\begin{equation}\label{8b}\forall x\in A, \ \mathcal{L}^{4}(A^{S}_{t})= \displaystyle{\int_{A^{S}_{t}} dz}=\displaystyle{ \int_{A} \mathcal{J}(t,z)\ dz}\end{equation}
and \begin{equation}\label{8c} \mathcal{L}^{4}(A^{S}_{t})= \displaystyle{ \int_{A} exp\Bigl(\int_{0}^{t} div\ X(\varphi^{X}_{\varepsilon s}(z))\ \ ds\Bigr) \ dz }\end{equation}
\end{lem}
\vspace*{1cm}

The following result is an immediate corollary of Lemma \ref{8}.

\begin{lem}\label{9}
Let $T>0$. For every subset $A$ in $\mathcal{V}$, we have 

\begin{equation}\label{9a}
\mathcal{L}^{4}(A^{S}_{t})\ \geq\ exp(-C\ l(A,t))\ \mathcal{L}^{4}(A),\ \forall t\in [0,T].
\end{equation}
\end{lem}

\begin{proof}[\textbf{Proof of Lemma \ref{9}}]

Let $A$ be a subset in $\mathcal{V}$.
 By Lemma \ref{7i}, there is a constant $C>0$ such that 
$$div\ X(z) \geq -C | X(z)|, \ \forall z\in A.$$ 
Therefore, by (\ref{8c}), we infer that, $ \forall t\in [0,T]$,
\begin{eqnarray*}
\mathcal{L}^{4}(A^{S}_{t}) &\geq& \displaystyle{ \int_{A} exp\Bigl(-C\int_{0}^{t} |X(\varphi^{X}_{\varepsilon s}(z))|\ \ ds\Bigr) \ dz }\\
&\geq& \displaystyle{ \int_{A} exp\Bigl(-C\ l(A,t)\Bigr) \ dz }\\
&\geq& \displaystyle{ exp\Bigl(-C\ l(A,t)\Bigr) \mathcal{L}^{4}(A)}.
\end{eqnarray*}
\end{proof}
\vspace*{0.5cm}

The following result whose proof is based on the local contraction property, is fundamental.
\begin{lem}\label{10}
Let $T>0$. The closed set given by
$$\{ x\in \mathcal{M};\ \exists \gamma \in \Omega^{S}_{x,T} \ \text{such that}\ \gamma(T) \in \mathcal{H}^{c} \}$$
is of Lebesgue measure zero on $M$.
\end{lem}

\begin{proof}[\textbf{Proof of Lemma \ref{10}}]

Let $A$ be a subset of $\mathcal{M}$ of positive Lebesgue measure. Without loss of generality, we can assume that $A$ is contained in an open set $\mathcal{V}$ in $M$. We argue by contradiction by assuming that $$\mathcal{L}^{4}(\{ x\in A;\ \exists \gamma \in \Omega^{S}_{x,T} \ \text{such that}\ \gamma(T) \in \mathcal{H}^{c} \}) >0.$$

By Lemma \ref{7}, there is an analytic horizontal vector field $X$ defined on $\mathcal{V}$ generating singular minimizing geodesic defined on $\mathcal{V}$.\\
\begin{center}
\begin{tikzpicture}[scale=0.7]
\draw (0, 2) ellipse (0.5cm and 2.5cm);
\draw (0,4.5) node[above]{$A$};
\draw (0,2) node {$\bullet$};
\draw (0,2) node[below]{$x$};
\draw (0, 2) .. controls (2,4) and (12.75, 2) .. (12.75, 3);
\draw (0, 3) .. controls (2,5) and (12.75, 3) .. (12.75, 4);
\draw (0, 1) .. controls (2,3) and (12.75, 1) .. (12.75, 2);
\draw (12,6) -- (13,5);
\draw (12,-0.5) -- (13,-1.5);
\draw (12,6) -- (12,-0.5);
\draw (13,-1.5) -- (13,5);
\draw (13, -0.5) node[right]{$\mathcal{H}^{c}$};
\draw [dashed](10, 2.6) ellipse (0.5cm and 2.5cm);
\draw (10,5.1) node[above]{$A^{S}_{t}$};
\draw (10,2.7) node {$\bullet$};
\draw (8.7,2.8) node[below]{$\varphi^{X}_{\varepsilon t}(x)$};
\end{tikzpicture}
\end{center}

Moreover, $X$ vanishes on $\mathcal{H}^{c}$. Then, for every $x\in A$, the flow of $X$ starting at $x$ requires an infinite time to reach $\mathcal{H}^{c}$, that is  
$$A^{S}_{t}= \varphi^{X}_{\varepsilon t}(A) \underset{t\to \infty}{\longrightarrow}S \subset \mathcal{H}^{c}.$$
Let $t \to \infty$, we obtain that $\mathcal{L}^{4}(A^{S}_{t})\longrightarrow 0$.\\
By Lemma (\ref{9}), we have $$\mathcal{L}^{4}(A^{S}_{t}) \geq\ exp(-C\ l(A,t)) \mathcal{L}^{4}(A),\ \forall t\in [0,T]. $$
By (\ref{len}), we obtain $$l(A,t) \leq L, \forall t\in [0, T].$$
Hence, $$ \mathcal{L}^{4}(A^{S}_{t}) \geq\ exp(-CL) \mathcal{L}^{4}(A),\ \forall t\in [0,T]. $$
When $t\to +\infty$, we obtain $$ \mathcal{L}^{4}(A)=0,$$
which implies the contradiction.
\end{proof}
\vspace*{1cm}

In the spirit of \cite{CH15}, we have the following result.
\begin{lem}\label{11}
Let $\Lambda_{1}$, $\Lambda_{2}$ be two subsets of $\Gamma$ such that 
\begin{enumerate}[(i)]
\item{$P^{1}(\Lambda_{1})= P^{1}(\Lambda_{2})$ and $P^{1}(\Lambda_{i}) \subset \mathcal{M}^{S}, \forall i=1,2$.}
\item{$P^{2}(\Lambda_{1})\cap P^{2}(\Lambda_{2}) = \emptyset$.}
\end{enumerate}
Then, $ \mathcal{L}^{4}(P^{1}(\Lambda_{1})) =\mathcal{L}^{4}(P^{1}(\Lambda_{2}))=0.$
\end{lem}

\begin{proof}[\textbf{Proof of Lemma \ref{11}}]

Set $A=P^{1}(\Lambda_{1})= P^{1}(\Lambda_{2})$. We can assume that $A$ is contained in an open set $\mathcal{V}$ in $M$. Let $T>0$. For every $i=1, 2$, we define
$$ A^{S, \Lambda_{i}}_{t}:=\{ \varphi^{X}_{\varepsilon t}(x) |\ \varphi^{X}_{0}(x)\in A \ \text{and}\ \varphi^{X}_{\varepsilon T}(x )\in P^{2}(\Lambda_{i})\}, \ \forall t\in [0, T].$$

Since $P^{2}(\Lambda_{1})\cap P^{2}(\Lambda_{2}) = \emptyset $, we have
$$A^{S,\Lambda_{1}}_{t} \cap A^{S,\Lambda_{2}}_{t}=\emptyset, \forall t \in [0, T].$$

For $\delta>0$ fixed, we define $A^{\delta}= \{ x : d_{SR}(x, A) \leq \delta \}$. \\

\begin{center}
\begin{tikzpicture}[scale=0.7]
\draw (-5, 4) ellipse (1cm and 2.5cm);
\draw (-5,6.5) node[below]{$A$};
\draw (-5, 4) .. controls (-4.5,4.5) and (10, 5.5) .. (10, 6);
\draw (-5, 4.5) .. controls (-4.5,5) and (10, 6) .. (10, 6.5);
\draw (-5, 3.5) .. controls (-4.5,4) and (10, 5) .. (10, 5.5);
\draw (10, 6) ellipse (1cm and 2.5cm);
\draw (10,3.5) node[below]{$P^{2}(\Lambda_{2})$};
\draw (6, 5.5) ellipse (1cm and 2.5cm);
\draw (6,3) node[below]{$P^{2}(\Lambda_{1})$};
\draw  [dashed](-2, 4.5) ellipse (1cm and 2.5cm);
\draw (-2, 2) node[below]{$A^{S,\Lambda_{1}}_{t}$};
\draw  [dashed](0.3,4.7) ellipse (1cm and 2.5cm);
\draw (0.3, 2.2) node[below]{$A^{S,\Lambda_{2}}_{t}$};
\draw [dashed] (-5, 4) ellipse (1.25cm and 2.75cm);
\draw (-5, 6.75) node[above]{$A^{\delta}$};
\end{tikzpicture}
\end{center}

$\mathcal{L}^{4}(A) =\displaystyle{  \lim_{\delta \to 0} \sup \mathcal{L}^{4}(A^{\delta}) }$\\

\hspace*{1.3cm}$\geq\displaystyle{  \lim_{t \to 0} \sup \mathcal{L}^{4}(A^{S,\Lambda_{1}}_{t}\cup A^{S,\Lambda_{2}}_{t})}$\\

\hspace*{1.3cm}$=\displaystyle{  \lim_{t \to 0} \sup [\mathcal{L}^{4}(A^{S,\Lambda_{1}}_{t})+ \mathcal{L}^{4} (A^{S,\Lambda_{2}}_{t})] }$\\

\hspace*{1.3cm}$\geq \displaystyle{ 2\  exp \Bigl(-C\ l(A,t)\Bigr) \mathcal{L}^{4}(A)}$.\\

Since $t\to 0$, we have $A_{t}^{S,\Lambda_{i}}$ very close to $A$. So we can choose \\ $l(A,t)>0$ sufficiently small, that is $$exp\Bigl(-C\ l(A,t) \Bigr) > \displaystyle{\frac{1}{2}}.$$
Hence, we obtain $\mathcal{L}^{4}(A) =0.$
\end{proof}
\vspace*{1cm}

We are ready to complete the proof of Proposition \ref{3}.\\
Consider the following set\\

\hspace*{3cm}$E:= \{ x\in \mathcal{M}^{S}: \Gamma^{S}(x) \ \text{ is not a singleton} \}$\\

and assume that $E$ has positive measure. It follows that there is $k\in \N$ such that the set given by\\ 

\hspace*{3cm}$E_{k}:=\{x\in E :\ diam\ \displaystyle{\Gamma^{S}(x)> \frac{1}{k}} \}$\\ has positive Lebesgue measure.\\

Without loss of generality, we can assume that the manifold $M$ can be covered by finitely many open balls $(\mathcal{U}_{i})_{i\in I}$ of diameter less or equal to $1/k$. From $(\mathcal{U}_{i})_{i\in I}$, we construct a finite family of open sets $(\mathcal{V}_{i})_{i\in I}$ pairwise disjoint covering $M$ by proceeding as follows \\

$\left\{\begin{array}{lcl}
\mathcal{V}_{1}& = & \mathcal{U}_{1}\\
\mathcal{V}_{2}& = & \mathcal{U}_{2}\textbackslash \mathcal{U}_{1}\\
&\vdots&\\
\mathcal{V}_{n}&=& \mathcal{U}_{n}\textbackslash (\mathcal{U}_{1}\cup \mathcal{U}_{2} \cup \dots\cup \mathcal{U}_{n-1})\\
&\vdots&\end{array}\right.$\\

such that $\displaystyle{\bigcup_{i\in I}\mathcal{U}_{i}=\bigcup_{i\in I}\mathcal{V}_{i}}$.

Therefore, for any $x\in E_{k}$, there are $i_{x}, j_{x}\in I$ with $i_{x}\neq j_{x}$ such that 
$$\Gamma^{S}(x) \cap \mathcal{V}_{i_{x}} \neq \emptyset \ \text{and}\ \Gamma^{S}(x) \cap \mathcal{V}_{j_{x}} \neq \emptyset.$$

Denote by 
$$E_{k,i}:= \bigcup_{x\in E_{k}} \{x\}\times(\Gamma^{S}(x)\cap \mathcal{V}_{i_{x}})$$
\hspace*{3cm}and 
$$E_{k,j}:= \bigcup_{x\in E_{k}} \{x\}\times(\Gamma^{S}(x)\cap \mathcal{V}_{j_{x}}).$$

We notice that $P^{1}(E_{k,i})=P^{1}(E_{k,j})= E$ such that 
\begin{equation}\label{con}
\mathcal{L}^{4}(E)>0.
\end{equation}
We also have $P^{2}(E_{k,i})\cap P^{2}(E_{k,j}) = \emptyset$ since for any $x\in E_{k}$, $\mathcal{V}_{i_{x}}\cap\mathcal{V}_{j_{x}}=\emptyset$, for $i_{x}\neq j_{x}$. Using lemma~\ref{11}, we obtain $\mathcal{L}^{4}(P^{1}(E_{k}))=0$, which contradicts assumption $(\ref{con})$.\\

We conclude that for a.e. $x\in \mathcal{M}^{S}, \Gamma^{S}(x)$ is a singleton.

\section{End of the proof of Theorem \ref{1}}\label{SECend}

In the previous sections, we have shown that 
$$\forall x\in \mathcal{M}^{R}, \Gamma^{R}(x)\ \text{is a singleton} \ (\text{see section}\ \ref{SECregular}),$$
and 
$$\forall x\in \mathcal{M}^{S}, \Gamma^{S}(x)\ \text{is a singleton} \ (\text{see section}\ \ref{SECsingular}).$$
\vspace*{0.5cm}

To complete the proof of Theorem \ref{1}, it remains to prove that 
$$\forall x\in \mathcal{M}^{S}\cap \mathcal{M}^{R}, \Gamma(x) \ \text{is a singleton}.$$
For this purpose, we will use again the technique introduced by Cavalletti and Huesmann \cite{CH15}. 
First, we will show a localized contraction property for regular horizontal curves.

\begin{lem}\label{12}
There is a positive constant $\tilde{C}$ such that for $T>0$ and 
for every set $A$ in $\mathcal{M}^{R}$,

\begin{equation}\label{12a} \mathcal{L}^{4}(A^{R}_{t})\geq \tilde{C}\mathcal{L}^{4}(A),\ \forall t\in [0, T]\end{equation}

with
$$A^{R}_{t}:=\{\gamma(t) |\ \gamma\in \Omega^{R}_{x,T};\ x\in A\ \text{and}\ \gamma(T)\in \Gamma^{R}(x)\}.$$
\end{lem}

\begin{proof}[\textbf{Proof of Lemma \ref{12}}]

Let $A$ be a compact set of $\mathcal{M}^{R}$ of positive measure. Since $\displaystyle{\mathcal{M}^{R}\subset \bigcup_{k\in \N} W_{k}}$ (by Lemma \ref{3}), for every point $x$ of $A$,  there exists \\$k=k(x)\in \N$ such that
$$x \in A_{k}:= A\cap W_{k},$$
so there is $p_{x}\in \R^{4}$ with $|p_{x}|\leq k$ verifying 
$$\varphi(x) \leq \varphi(z) - <p_{x}, x-z > + k |x- z|^{2}, \ \forall z\in B(x,1/k).$$
Let $\tilde{A}_{k}:= A_{k}\cap B(x, 1/2k)$. As in section \ref{SECregular}, we define the function  $$\tilde{\varphi}(z) = \left \{ \begin{array}{lcl}\varphi(z) & \text{if}\ z\in \tilde{A}_{k}\\
\vspace*{0.05cm}\\
\displaystyle{\sup_{y\in \tilde{A}_{k}} \{  \varphi(y) + <p_{y}, y-z> - k\ |y-z|^{2} \}}& \text{if not}\end{array}\right.$$\\
For any $x \in A$, $\tilde{\varphi}$ is locally semiconvex on $B(x,1/2k)$. By the Alexandrov Theorem, $\tilde{\varphi}$ is twice differentiable at a.e. $z\in B(x,1/2k)$. Moreover, there exists a constant $C_{k}>0$ such that 
\begin{equation}\label{12a}Hess_{z}\tilde{\varphi} \geq - C_{k} I_{4},\ a.e.\ z\in B(x,1/2k)\end{equation}
where $I_{4}$ is the $4\times 4$ identity matrix.\\

We notice that $A=\displaystyle{ \bigcup_{k\in \N} \tilde{A}_{k}} $. Denote by $ \tilde{C}>0$ the constant given by
$$\tilde{C}:= \displaystyle{ \sup_{k\in \N} C_{k}}.$$
Then, 
$$Hess_{x} \tilde{\varphi} \geq - \tilde{C} I_{4},\ a.e.  x\in A.$$

By section \ref{SECregular}, for almost every $x\in A\subset\mathcal{M}^{R}$, there exists a unique $y\in \Gamma^{R}(x)$ given by
$$y:= exp_{x}(-d_{x}\tilde{\varphi}).$$
Then, the curve $\gamma_{x}(t): [0, T]\rightarrow M$ defined by
$$\displaystyle{\gamma_{x}(t):= exp_{x}(-td_{x}\tilde{\varphi})}, \ a.e. \ x\in A$$
is the unique regular minimizing geodesic joining $x$ to $y$.\\

For every $t\in [0, T]$, we define the function 
$$\begin{array}{lccl}
T_{t}:& M& \rightarrow & M \\
&x& \mapsto & T_{t}(x)= \gamma_{x}(t) = exp_{x}(-td_{x}\tilde{\varphi})
\end{array}.$$

Note that, $\forall t\in [0,T]$, $A^{R}_{t}= \{T_{t}(z) : z\in A \}$ then we have 
\begin{equation}\label{Reg}\displaystyle{\mathcal{L}^{4}(A^{R}_{t})=\int_{A^{R}_{t}}\mathrm{d}x = \int_{\{T_{t}(z); z\in A\}}\mathrm{d}x=\int_{A}det( Jac\ T_{t}(x))\mathrm{d}x}.\end{equation}

However, the function $T_{t}$ results from the composition of the two following functions
$$f:x\in M \rightarrow d_{x}\tilde{\varphi}\in T^{*}_{x}M,\ \text{and}\ g: p\in T^{*}M\rightarrow exp_{x}(-tp) \in M.$$

By computing the Jacobien of $T_{t}$, we obtain

$$\displaystyle{Jac\ T_{t}(x)= Jac\ g(f(x))\times Hess_{x}\tilde{\varphi}  }\ .$$

Here, $g$ is smooth on $T^{*}M$ and by (\ref{12a}), there is a constant $\tilde{C}>0$ such that 
$$Jac\ T_{t}(x) \geq -\tilde{C}\ I_{4},\ a.e.\ x\in A.$$

By (\ref{Reg}), this implies
 
$$\mathcal{L}^{4}(A^{R}_{t}) \geq \displaystyle{ \tilde{C} \mathcal{L}^{4}(A)}, \forall t\in [0, T].$$
\end{proof}
\vspace*{0.5cm}

We conclude with the following lemma.
\begin{lem}\label{14}
$\mathcal{M}^{R}\cap \mathcal{M}^{S}$ has Lebesgue  measure zero on $M$.\end{lem}

\begin{proof}[\textbf{Proof of Lemma \ref{14}}]
Assume that there is a set $A$ of $\mathcal{M}^{R}\cap \mathcal{M}^{S}$ such that 
\begin{equation}\label{14a}\mathcal{L}^{4}(A)>0.\end{equation}

Let $T>0$ and $\varepsilon \in \{-1,+1\}$. For every $t\in [0,T]$, we define the two following intermediate subsets by
$$A^{R}_{t}:=\{\gamma_{x}(t) |\ \gamma_{x}\in \Omega^{R}_{x,T}\ \text{with}\  x\in A\ \text{and}\ \gamma^{R}_{x}(T)\in \Gamma^{R}(x) \},  $$
and 
$$A^{S}_{t}:=\varphi^{X}_{\varepsilon t}(A).$$
\vspace*{0.2cm}

For every $x\in A$, we have $\Gamma^{R}(x)\cap \Gamma^{S}(x)= \emptyset$, then there is $t=t(x) \in ]0,T[$ such that $$\varphi^{X}_{\varepsilon s}(x) \neq \gamma_{x}(s), \forall s\in ]t,T].$$ 

As a matter of fact, regular minimizing geodesics are analytic as projections of the analytic sub-Riemannian Hamiltonian system and singular minimizing geodesic are analytic as the analytic flow of $X$. Without loss of generality, we can assume that there is $\bar{t}\in ]0,1[$ such that for every $x\in A$ 
$$t=t(x) \leq \bar{t}\ \text{and}\ A^{R}_{s}\cap A^{S}_{s}= \emptyset, \ \forall s\in ]\bar{t},T]$$
and
$$\ A^{R}_{\bar{t}}\cap A^{S}_{\bar{t}}\neq \emptyset.$$
\vspace*{0.2cm}

We denote by $$\bar{A}:= A^{R}_{\bar{t}}\cup A^{S}_{\bar{t}}.$$ We may assume that $\bar{A}$ has positive Lebesgue measure. Notice that for $s\geq \bar{t}$,  when $s\to\bar{t}$, $A_{s}^{R}$ and $A_{s}^{S}$ converge to $\bar{A}$, then one has \\

$\mathcal{L}^{4}(\bar{A}) =\displaystyle{  \lim_{\delta \to 0} \sup \mathcal{L}^{4}(\bar{A}^{\delta}) }\geq\displaystyle{ \lim_{s \to \bar{t}^{+}} \sup \mathcal{L}^{4}(A^{\Lambda_{1}}_{s}\cup A^{\Lambda_{2}}_{s})}$\\

\hspace*{4.3cm}$=\displaystyle{  \lim_{s \to \bar{t}^{+}} \sup \mathcal{L}^{4}(A^{R}_{s}\cup A^{S}_{s})}$\\

\hspace*{4.3cm}$=\displaystyle{  \lim_{s \to \bar{t}^{+}} \sup [\mathcal{L}^{4}(A^{R}_{s})+ \mathcal{L}^{4} (A^{S}_{s})] }$

\begin{equation}\label{eq}\hspace*{3cm}\geq \displaystyle{ \lim_{s \to \bar{t}^{+}} \Big( exp \Bigl(-C\ l(\bar{A},s)\Bigr) + \tilde{C}\Big) \mathcal{L}^{4}(\bar{A})}.\end{equation}

where $\bar{A}^{\delta}:=\{x ; d_{SR}(x,\bar{A})\leq \delta \}$, for a given $\delta>0$.\\

 The inequality (\ref{eq}) follows from Lemmas \ref{9} and \ref{12} according to which we have
 $$\displaystyle{ \mathcal{L}^{4}(A^{R}_{s})\geq \tilde{C} \mathcal{L}^{4}(\bar{A})}\ \text{and}\ \mathcal{L}^{4} (A^{S}_{s})\geq exp \Bigl(-Cl(\bar{A},s)\Bigr) \mathcal{L}^{4}(\bar{A}), \forall s\in ]\bar{t},T[.$$
As $s\to \bar{t}$, we can choose $l(\bar{A},s)>0$ sufficiently small, that is $$exp\Bigl(-C\ l(\bar{A},s)\Bigr)+ \tilde{C} > 1.$$
It implies that $\mathcal{L}^{4}(\bar{A}) =0.$ And the conclusion follows.
\end{proof}

\appendix

\section{Proof of Lemma \ref{PROPstatic}}\label{A}

For every $y\in M$, the function $z\in M\mapsto \psi(y)- d_{SR}^{2}(z,y)$ is locally Lipschitz with respect to the sub-Riemannian distance. Then, $\forall z\in M$, 
$$\displaystyle{\varphi(z)=\sup_{y \in M}\{\psi(y)- d_{SR}^{2}(z,y)\}}$$
 is also locally Lipschitz with respect to $d_{SR}^{2}$. \\

 Fix $x\in M$, there are an open neighborhood $\mathcal{V}$ of $x$ and an orthonormal family of $m$ vector fields $X^{1}, \dots, X^{m}$ such that $\Delta(z) = span\{X^{1}(z), \dots, X^{m}(z) \}$, $\forall z \in \mathcal{V}$. By a change of coordinates if necessary, we can write the vector fields as the following form:
$$ X^{i} = \displaystyle{ \frac{\partial}{\partial x_{i}} + \sum_{j=1}^{n}a_{ij} \frac{\partial}{\partial x_{j}}}, \ \forall i=1,\dots,m.$$
By the Pansu-Rademacher theorem, since $\mu$ is absolutely continuous with respect to the Lebesgue measure, then $\varphi$ is differentiable with respect to the vector fields $X^{1}, \dots, X^{m}$, $ \mu-a.e.\ z\in \mathcal{V}$. Hence, we have: $$\varphi(y) - \varphi(x) = \sum_{i=1}^{m} X^{i}\varphi(x)(y_{i}-x_{i})+ o(d_{SR}(x,y)), \ \forall y \in \mathcal{V}.$$
Let $\gamma^{x}_{i}: [0,1]\rightarrow M$, $i=1, \dots,m$ be the integral flow associated to $X^{i}$ starting at $x$. Then, 

$$\displaystyle{\lim_{t\to 0}\frac{\varphi(\gamma^{x}_{i}(t))-\varphi(x)}{t}=l_{i}, \forall i=1, \dots,m}.$$

Recall that $g(\gamma^{x}_{i}(t), \gamma^{x}_{i}(t))= g(X^{i}(\gamma^{x}_{i}(t)),X^{i}(\gamma^{x}_{i}(t)))=1$, $\forall t\in [0,1]$. \\

Then, $d_{SR}(x, \gamma^{x}_{i}(t)) \leq |t|$, $ \forall t\in [0,1]$.\\

$x\in \Gamma(x) \Rightarrow \varphi(x) -\varphi(z) \leq d_{SR}^{2}(x,z), \forall z\in \mathcal{V}$.\\

In particular, $\varphi(x) -\varphi( \gamma^{x}_{i}(t)) \leq d_{SR}^{2}(x, \gamma^{x}_{i}(t)) \leq t^{2}$.\\

This implies that $l_{i}=0$. Hence, $X^{i}\varphi(x)=0$, $\forall i=1, \dots, m$.\\

Assume now that there exists $y \in \Gamma(x)$ such that $y\neq x$. Let\\ $\gamma_{x,y}: [0,1]\rightarrow M$ be a minimizing geodesic joining $x$ to $y$.\\

$\varphi(x)- \varphi(z) \leq d_{SR}^{2}(x,z) - d_{SR}^{2}(x,y), \forall z\in \mathcal{V}$\\

$ \forall t\in[0,1]$, $\varphi(x)- \varphi(\gamma_{x,y}(t))\leq d_{SR}^{2}(x,\gamma_{x,y}(t)) - d_{SR}^{2}(x,y)$,\\

$\hspace*{2cm}\Rightarrow -o(d_{SR}(x,\gamma_{x,y}(t))) \leq d_{SR}^{2}(x,\gamma_{x,y}(t)) - d_{SR}^{2}(x,y),$\\

$\hspace*{2cm}\Rightarrow -o(t\ d_{SR}(x,y))\leq (1-t)^{2}d_{SR}^{2}(x,y) - d_{SR}^{2}(x,y),$\\

$\hspace*{2cm}\Rightarrow -o(t\ d_{SR}(x,y)) \leq -2t\ d_{SR}^{2}(x,y)+ t^{2}\ d_{SR}^{2}(x,y),$\\

$\hspace*{2cm}\Rightarrow o(t\ d_{SR}(x,y)) \geq 2t\ d_{SR}^{2}(x,y)- o(t\ d_{SR}(x,y)),$\\

$\hspace*{2cm}\Rightarrow o(t\ d_{SR}(x,y))\geq t\ d_{SR}^{2}(x,y).$\\

For $t$ small enough, there is a contradiction since $x\neq y$.
\begin{flushright} $\square$ \end{flushright}

\section{Local semiconvexity}\label{B}
Let $(\Delta, g)$ be a sub-Riemannian structure of rank $m\leq n$ on the manifold $M$.\\

We recall here the definition of local semiconvexity of a given function.

\begin{defi}
A function $f:\Omega\rightarrow \R$, defined on the open set $\Omega\subset M$, is called locally semiconvex on $\Omega$ if for every $x\in \Omega$ there exist a neighborhood $\Omega_{x}$ of $x$ and a smooth diffeomorphism $\varphi_{x}:\Omega_{x}\rightarrow \varphi_{x}(\Omega_{x})$ such that $f\circ \varphi_{x}^{-1}$ is locally semiconvex on the open subset $\tilde{\Omega}_{x}= \varphi_{x}(\Omega_{x})\subset \R^{n}$. \\

By the way, we recall that the function $\tilde{f}: \tilde{\Omega}\rightarrow \R$ is locally semiconvex on the open subset $\tilde{\Omega}\subset \R^{n}$  if for every $\bar{x}\in \tilde{\Omega}$ there exist $C, \delta>0$ such that
$$f\Big(\lambda x+ (1-\lambda)y\Big )-\lambda f(x) -(1-\lambda)f(y) \leq \lambda(1-\lambda) C|x-y|^{2},$$
$$\forall \lambda\in [0, 1], \forall x,y\in B(\bar{x}, \delta)$$
where $B(\bar{x}, \delta)$ is the open ball in $\R^{n}$ centered at $\bar{x}$ with radius $\delta$.
\end{defi}

The following result is useful to prove the local semiconvexity of a given function.

\begin{lem}\label{lemap}
Let $f:\Omega \rightarrow \R$ be a function defined on an open set $\Omega\subset \R^{n}$. Assume that for every $\bar{x}\in \Omega$, there exist a neighborhood $\mathcal{V}\subset \Omega$ of $\bar{x}$ and a positive real number $\sigma$ such that, for every $x\in \mathcal{V}$, there is $p_{x}\in \R^{n}$ such that 
$$f(x)\leq f(y) - <p_{x}, x- y > + \sigma |x-y|^{2},\ \forall y\in \mathcal{V}.$$
Then, the function $f$ is locally semiconvex on $\Omega$.
\end{lem}

\begin{proof}[\textbf{Proof of Lemma \ref{lemap}}]

Let $\bar{x}\in \Omega$ be fixed and $\mathcal{V}$ be the neighborhood given by assumption. Without loss of generality, we can assume that $\mathcal{V}$ is an open ball $\mathcal{B}$. Let $x,y \in \mathcal{B}$ and $\lambda \in [0,1]$. The point $\hat{x}:= \lambda x+ (1-\lambda) y$ belongs to $\mathcal{B}$. By assumption, 
there exists $\hat{p}\in \R^{n}$ such that 
$$f(\hat{x}) \leq f(z) - <\hat{p}, \hat{x}- z > + \sigma |\hat{x}- z|^{2},\ \forall z\in \mathcal{B}.$$
Hence, we easily get
$$\left \{  \begin{array}{lcl}
\vspace*{0.4cm}
f(\hat{x}) &\leq& f(x) - (1-\lambda) <\hat{p}, y- x > + \sigma (1-\lambda) |x-y|^{2}\\

f(\hat{x}) &\leq& f(y) - \lambda <\hat{p}, x-y > + \sigma \lambda |x-y|^{2}
\end{array} \right.$$

$$\hspace*{0.8cm}\Rightarrow \left \{  \begin{array}{lcl}
\vspace*{0.4cm}
\lambda f(\hat{x})& \leq& \lambda f(x) +\lambda (1-\lambda) <\hat{p}, x-y > + \sigma \lambda (1-\lambda) |x-y|^{2}\\

(1-\lambda) f(\hat{x}) &\leq& (1-\lambda) f(y) - \lambda (1-\lambda) <\hat{p}, x-y > + \sigma \lambda(1-\lambda) |x-y|^{2}
\end{array} \right.$$

$$\Rightarrow f(\hat{x}) \leq \lambda f(x)+ (1- \lambda) f(y) + 2 \lambda (1- \lambda) \sigma |x-y|^{2}$$

and the conclusion follows.
\end{proof}
\vspace*{0.5cm}

\begin{rmk}
Thanks to Lemma \ref{lemap}, a way to prove that a given function \\ $f:\Omega \rightarrow \R$ is locally semiconvex on $\Omega$ is to show that for every $x\in \Omega$, we can put a support function $\phi$ of class $C^{2}$ under the graph of $f$ at $x$ with a uniform control of $C^{2}$ norm of $\phi$.
\end{rmk}
\vspace*{0.5cm}

Let us derive another important consequence of the definition of semiconvexity.

\begin{lem}\label{lemap2}
Let $\Omega$ be a subset of $\R^{n}$ and $\{u_{\alpha}\}_{\alpha\in \mathcal{A}}$ be a family of functions defined on $\Omega$ and semiconvex. Then, the function $u:=\displaystyle{\sup_{\alpha\in \mathcal{A}}u_{\alpha}}$ is also semiconvex on $\Omega$. 
\end{lem}

\begin{proof}[\textbf{Proof of Lemma \ref{lemap2}}]
Take $x,y \in \Omega$ and $\lambda \in [0, 1]$. \\
Given any $\varepsilon >0$, we can find $\alpha$ such that
$$u(\lambda x + (1- \lambda) y) \leq u_{\alpha}(\lambda x + (1- \lambda) y) + \varepsilon.$$

Then we have, for $C_{\alpha}, \delta_{\alpha}>0$,

$$u(\lambda x + (1- \lambda) y) - \lambda u(x) - (1-\lambda) u(y)$$

$$\leq u_{\alpha}(\lambda x + (1- \lambda) y)+ \varepsilon - \lambda u_{\alpha}(x) - (1-\lambda) u_{\alpha}(y)$$

$$\leq \lambda(1-\lambda) C_{\alpha} |x-y|^{2}+ \varepsilon, \forall y\in B(x, \delta_{\alpha}).$$

Since $\varepsilon>0$ is arbitrary, we obtain the assertion.
\end{proof}
\vspace*{0.5cm}

More details of local semiconvexity of a given function are given in the textbook \cite{CS04}.

\end{document}